\documentclass[a4paper,10pt]{article}

\usepackage{authblk}
\usepackage{amsmath}
\usepackage{amssymb}
\usepackage{amsthm}
\usepackage[mathscr]{eucal}
\usepackage{mathtools}
\usepackage{bm}
\usepackage{bbm}
\usepackage{a4wide}
\usepackage{enumitem}
\usepackage{hyperref}
\usepackage{booktabs}

\usepackage{tikz}
\usepackage{graphicx}
\usepackage{subcaption}

\newcommand{\sss}{\scriptscriptstyle}

\newcommand{\prob}{\mathbb{P}}
\newcommand{\Prob}[1]{\prob\left(#1\right)}
\newcommand{\Probn}[1]{\prob_n\left(#1\right)}

\newcommand{\expec}{\mathbb{E}}
\newcommand{\Exp}[1]{\expec\left[#1\right]}
\newcommand{\Expn}[1]{\expec_n\left[#1\right]}

\newcommand{\Var}[1]{\textup{Var}\left(#1\right)}
\newcommand{\Varn}[1]{\textup{Var}_n\left(#1\right)}

\newcommand{\plim}{\ensuremath{\stackrel{\prob}{\longrightarrow}}}

\newcommand{\ind}[1]{\mathbbm{1}_{\left\{#1\right\}}}

\newcommand{\bigO}[1]{O\left(#1\right)}
\newcommand{\bigOp}[1]{O_{\sss\prob}\left(#1\right)}

\newcommand{\abs}[1]{\left|#1\right|}
\newcommand{\me}{\textup{e}}
\newcommand{\dd}{{\rm d}}

\newtheorem{theorem}{Theorem}[section]

\newtheorem{lemma}[theorem]{Lemma}
\newtheorem{proposition}[theorem]{Proposition}

\newcommand{\op}{o_{\sss\prob}}

\newcommand{\bfdit}{\boldsymbol{d}}
\newcommand{\bfDit}{\boldsymbol{D}}
\newcommand{\Der}{D^{\sss\mathrm{(er)}}}

\newcommand{\e}{{\mathrm e}}
\newcommand{\Nn}{N^{\sss(n)}}
\newcommand{\Mn}{M^{\sss(n)}}
\newcommand{\indic}[1]{\mathbbm{1}_{\{#1\}}}

\newcommand{\CMnd}{{\rm CM}_n(\boldsymbol{d})}

\allowdisplaybreaks

\graphicspath{{chfigures/}}
\setenumerate{label={\normalfont{(\roman*)}}} 
\SetLabelAlign{LeftAlignWithIndent}{\hspace*{1.0ex}\makebox[1.25em][l]{#1}}
\numberwithin{equation}{section}

\title{Triadic closure in configuration models\\ with unbounded degree fluctuations}
\author[1]{Remco van der Hofstad}
\author[1]{Johan S.H. van Leeuwaarden}
\author[1]{Clara Stegehuis}
\affil[1]{Department of Mathematics and Computer Science, Eindhoven University of Technology}

\begin{document}
	
	\maketitle
	
	\begin{abstract}
The configuration model generates random graphs with any given degree distribution, and thus serves as a null model for scale-free networks with power-law degrees and unbounded degree fluctuations. For this setting, we study the local clustering $c(k)$, i.e., the probability that two neighbors of a degree-$k$ node are neighbors themselves. We show that $ c(k)$ progressively falls off with $k$ and eventually for $k=\Omega(\sqrt{n})$ settles on a power law $c(k)\sim k^{-2(3-\tau)}$ with $\tau\in(2,3)$ the power-law exponent of the degree distribution. This fall-off has been observed in the majority of real-world networks and signals the presence of modular or hierarchical structure.
Our results agree with recent results for the hidden-variable model~\cite{stegehuis2017} and also give the expected number of triangles in the configuration model when counting triangles only once despite the presence of multi-edges. We show that only triangles consisting of triplets with uniquely specified degrees contribute to the triangle counting. 
	\end{abstract}
	
	\section{Introduction}
	\label{sec-intro}
	Random graphs can be used to model many different types of networked structures such as communication networks, social networks and biological networks. Many of these real-world networks display similar characteristics. A well-known characteristic of many real-world networks is that the degree distribution follows a power law. Another such property is that they are highly clustered. 
	Several statistics to measure clustering exist. The global clustering coefficient measures the fraction of triangles in the network. A second measure of clustering is the \emph{local clustering coefficient}, which measures the fraction of triangles that arise from one specific node.
	
	The local clustering coefficient $c(k)$ of vertices of degree $k$ decays when $k$ becomes large in many real-world networks. In particular, the decay was found to behave as an inverse  power of $k$ for $k$ large enough, so that $c(k)\sim k^{-\gamma}$ for some $\gamma>0$~\cite{vazquez2002,ravasz2003,serrano2006b,catanzaro2004,leskovec2008}, where most real-world networks were found to have $\gamma$ close to one. Figure~\ref{fig:ckex} shows the local clustering coefficient for a technological network (the Google web graph~\cite{snap}), an information network (hyperlinks of the online encyclopedia Baidu~\cite{niu2011}) and a social network (friends in the Gowalla social network~\cite{snap}). We see that for small values of $k$, $c(k)$ decays slowly. When $k$ becomes larger, the local clustering coefficient indeed seems to decay as an inverse power of $k$. Similar behavior has been observed for more real-world networks \cite{stegehuis2017}. The decay of the local clustering coefficient $c(k)$ in $k$ is considered an important empirical observation, because it signals the presence of hierarchical network structure~\cite{ravasz2003}, where high-degree vertices connect groups of clustered small-degree vertices.
	\begin{figure}[tb]
		\centering
		\begin{subfigure}{0.32\linewidth}
			\centering
			\includegraphics[width=\textwidth]{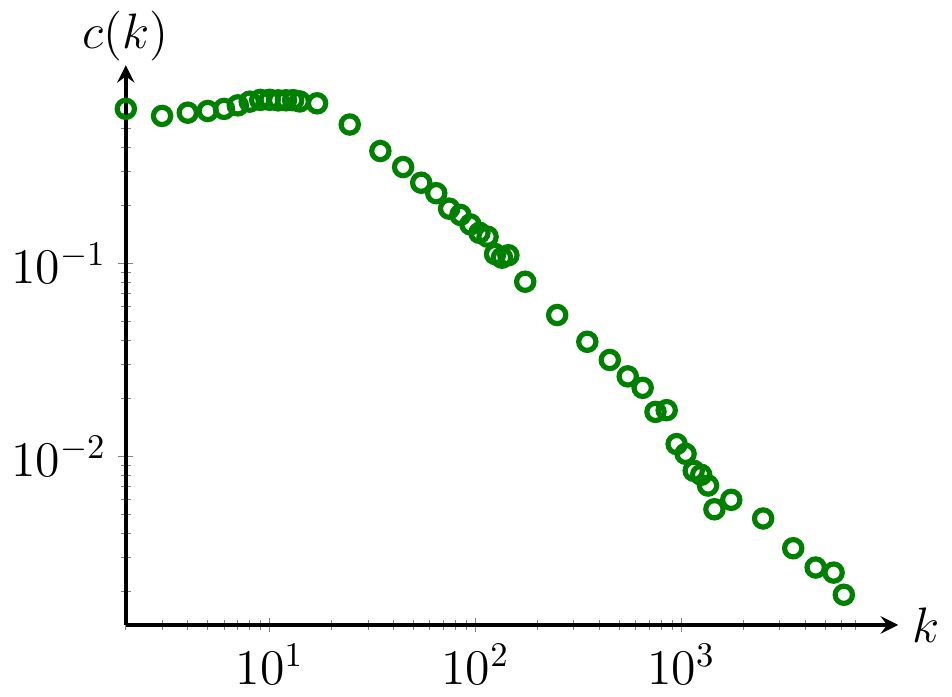}
			\caption{Google web graph~\cite{snap}}
		\end{subfigure}
	\begin{subfigure}{0.32\linewidth}
		\centering
		\includegraphics[width=\textwidth]{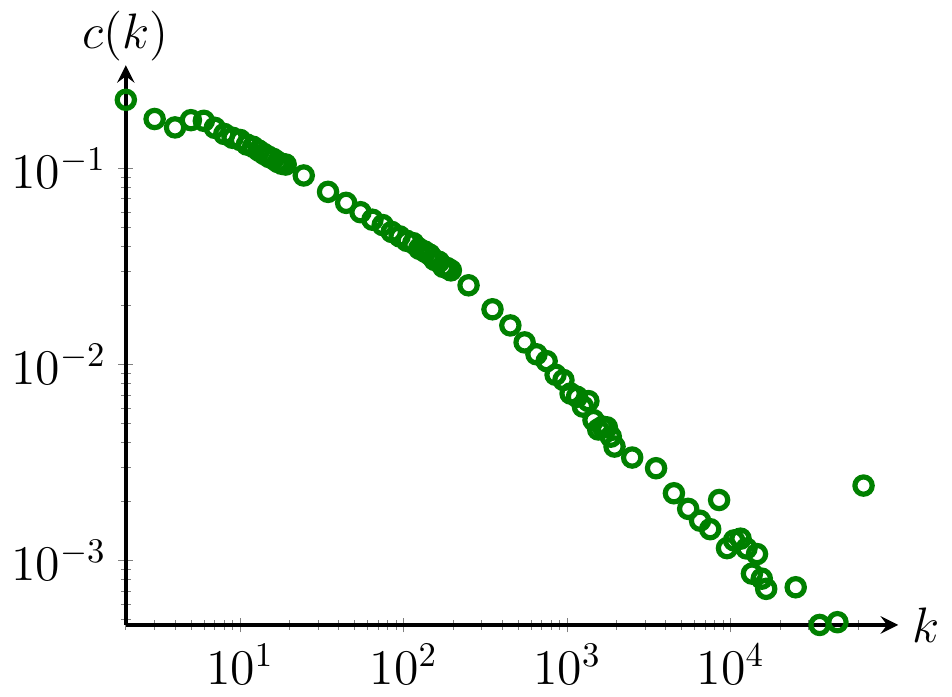}
		\caption{Baidu online encyclopedia~\cite{niu2011}}
\end{subfigure}
\begin{subfigure}{0.32\linewidth}
	\centering
	\includegraphics[width=\textwidth]{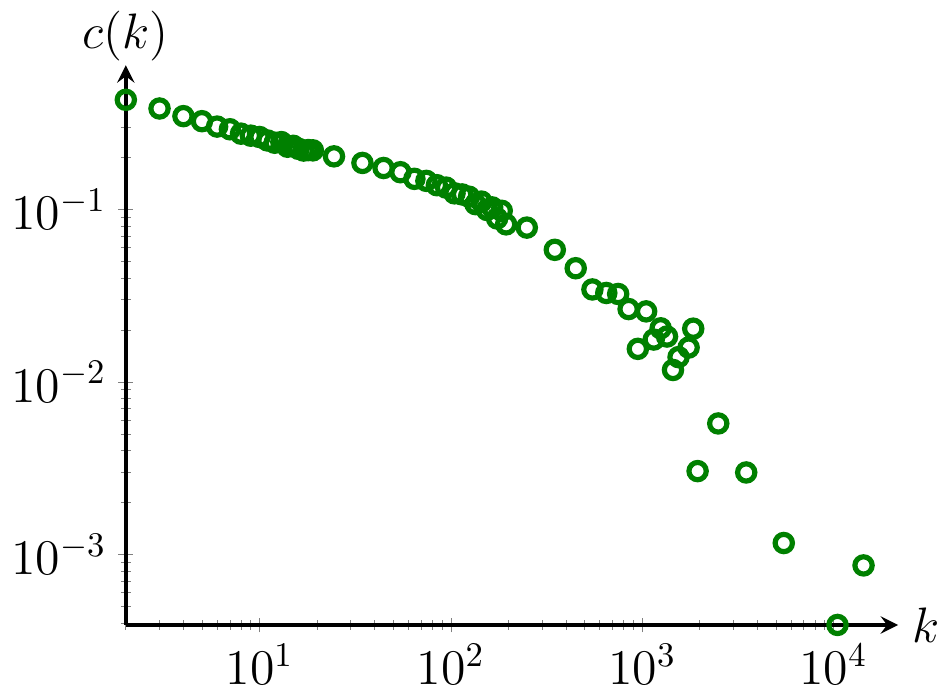}
	\caption{Gowalla social network~\cite{snap}}
\end{subfigure}
		\caption{Local clustering coefficient $c(k)$ for three real-world networks.}
		\label{fig:ckex}
	\end{figure}

	In this paper we analyze $c(k)$ for networks with a power-law degree distribution with degree exponent $\tau\in(2,3)$, the situation that describes the majority of real-world networks~\cite{albert1999,faloutsos1999,jeong2000,stegehuis2017}. To analyze $c(k)$, we consider the configuration model in the large-network limit, and count the number of triangles where at least one of the vertices has degree $k$. When the degree exponent $\tau>3$, the total number of triangles in the configuration model converges to a Poisson random variable~\cite[Chapter 7]{hofstad2009}. When $\tau\in(2,3)$, the configuration model consists of many self-loops and multiple edges~\cite{hofstad2009}. This creates multiple ways of counting the number of triangles, as we will show below. In this paper, we count the number of triangles from a \textit{vertex} perspective, which is the same as counting the number of triangles in the erased configuration model, where all self-loops have been removed and multiple edges have been merged.
	
We show that the local clustering coefficient remains a constant times $n^{2-\tau}\log(n)$ as long as $k\ll \sqrt{n}$. After that, $c(k)$ starts to decay as $c(k)\sim k^{-\gamma}n^{5-2\tau}$. We show that this exponent $\gamma$ depends on $\tau$ and can be larger than one. In particular, when the power-law degree exponent $\tau$ is close to two, the exponent $\gamma$ approaches two, a considerable difference with the preferential attachment model or several fractal-like random graph models that predict $c(k)\sim k^{-1}$~\cite{krot2015,ravasz2003,dorogovtsev2002}. Related to this result on the $c(k)$ fall-off, we also show that for every node with fixed degree $k$ only pairs of nodes with specific degrees contribute to the triangle count and hence local clustering. 

The paper is structured as follows. Section~\ref{sec:basic} contains a detailed description of the configuration model and the triangle count. We present our main results in Section~\ref{sec:main}, including Theorem~\ref{thm:ck} that describes the three ranges of $c(k)$. The remaining sections prove all the main results, and in particular focus on establishing  Propositions~\ref{prop:major} and~\ref{prop:minor} that are crucial for the proof of Theorem~\ref{thm:ck}.


\section{Basic notions}\label{sec:basic}
		
	\paragraph{Notation.}\label{sec:notation}
	We use $\overset{d}\longrightarrow$ for convergence in distribution, and $\plim $ for convergence in probability. We say that a sequence of events $(\mathcal{E}_n)_{n\geq 1}$ happens with high probability (w.h.p.) if $\lim_{n\to\infty}\Prob{\mathcal{E}_n}=1$. Furthermore, we write $f(n)=o(g(n))$ if $\lim_{n\to\infty}f(n)/g(n)=0$, and $f(n)=O(g(n))$ if $|f(n)|/g(n)$ is uniformly bounded, where $(g(n))_{n\geq 1}$ is nonnegative. Similarly, if $\limsup_{n\to\infty}\abs{f(n)}/g(n)>0$, we say that $f(g)=\Omega(g(n))$ for nonnegative $(g(n))_{n\geq 1}$. We write $f(n)=\Theta(g(n))$ if $f(n)=O(g(n) )$ as well as $f(n)=\Omega(g(n))$. We say that $X_n=O_{\sss{\prob}}(g(n))$ for a sequence of random variables $(X_n)_{n\geq 1}$ if $|X_n|/g(n)$ is a tight sequence of random variables, and $X_n=o_{\sss{\prob}}(g(n))$ if $X_n/g(n)\plim 0$.
	
	\paragraph{The configuration model.}
	Given a positive integer $n$ and a \emph{degree sequence}, i.e., a sequence of $n$ positive integers $\bfdit=(d_1,d_2,\ldots, d_n)$, the \emph{configuration model} is a (multi)graph where vertex $i$ has degree $d_i$. It is defined as follows, see e.g., \cite{bollobas2001} or \cite[Chapter 7]{hofstad2009}: Given a degree sequence $\bfdit$  with $\sum_{i\in[n]} d_i$ even, we start with $d_j$ free half-edges adjacent to vertex $j$, for $j=1, \ldots, n$. The random multigraph $\CMnd$ is constructed by successively pairing, uniformly at random, free half-edges into edges, until no free half-edges remain. (In other words, we create a uniformly random matching of the half-edges.) The wonderful property of the configuration model is that, conditionally on obtaining a simple graph, the resulting graph is a {\em uniform} graph with the prescribed degrees. This is why $\CMnd$ is often used as a {\em null model} for real-world networks with given degrees. 
	
	In this paper, we study the setting where the degree distribution has {\em infinite variance}. Then the number of self-loops and multiple edges tends to infinity in probability (see e.g., \cite[Chapter 7]{hofstad2009}), so that the configuration model results in a multigraph with high probability. 
	%
		In particular, we take the degrees $\bfdit$ to be an i.i.d.\ sample of a random variable $D$ such that
	\begin{equation}
		\label{D-tail}
		\prob(D=k)=C k^{-\tau}(1+o(1)),
	\end{equation}
	when $k\rightarrow \infty$, where $\tau\in(2,3)$ so that $\expec[D^2]=\infty$. When this sample constructs a sequence such that the sum of the variables is odd, we add an extra half-edge to the last vertex to obtain the degree sequence. This does not affect our computations. In this setting, $d_{\max}=\bigOp{n^{1/(\tau-1)}}$, where $d_{\max}=\max_{v\in[n]} d_v$ denotes the maximal degree of the degree sequence.   
	
	\paragraph{Counting triangles.}
	Let $G=(V,E)$ denote a configuration model with vertex set $V=[n]:=\{1,\ldots, n\}$ and edge set $E$. We are interested in the number of triangles in $G$. There are two ways to count triangles in the configuration model. The first approach is from an \emph{edge perspective}, as illustrated in Figure~\ref{fig:CMECM}. This approach counts the number of triples of edges that together create a triangle. This approach may count multiple triangles between one fixed triple of vertices. Let $X_{ij}$ denote the number of edges between vertex $i$ and $j$. Then, from an edge perspective, the number of triangles in the configuration model is
	\begin{equation}
		\sum_{1\leq i<j < k\leq n }X_{ij}X_{jk}X_{ik}.
	\end{equation}
	A different approach is to count the number of triangles from a \emph{vertex perspective}. This approach counts the number of triples of vertices that are connected. Counting the number of triangles in this way results in
	\begin{equation}
		\sum_{1\leq i<j < k\leq n }\ind{X_{ij}\geq 1}\ind{X_{jk}\geq 1}\ind{X_{ik}\geq 1}.
	\end{equation}
	When the configuration model results in a simple graph, these two approaches give the same result. When the configuration model results in a multigraph, these two approaches may give very different numbers of triangles. In particular, when the degree distribution follows a power-law with $\tau\in(2,3)$, the number of triangles is dominated by the number of triangles between the vertices of the highest degrees, even though only few such vertices are present in the graph~\cite{newman2010}. When the exponent $\tau$ of the degree distribution approaches 2, then the number of triangles between the vertices of the highest degrees will be as high as $n^3$, which is much higher than the number of triangles we would expect in any real-world network of that size. When we count triangles from a vertex perspective, we count only one triangle between these three vertices. Thus, the number of triangles from the vertex perspective will be significantly lower. In this paper, we focus on the vertex based approach for counting triangles. Note that this approach is the same as counting triangles in the \emph{erased configuration model}, where all multiple edges have been merged, and the self-loops have been removed.
	
	Let $\triangle_k$ denote the number of triangles attached to vertices of degree $k$. Note that when a triangle consists of two vertices of degree $k$, it is counted twice in $\triangle_k$.  Let $N_k$ denote the number of vertices of degree $k$.
	Then, the clustering coefficient of vertices with degree $k$ equals
	\begin{equation}
		\label{c(k)-def}
		c(k)=
		\frac{1}{N_k}\frac{2\triangle_k}{k(k-1)}.
	\end{equation}
	When we count $\triangle_k$ from the vertex perspective, this clustering coefficient can be interpreted as the probability that two random connections of a vertex with degree $k$ are connected. This version of $c(k)$ is the local clustering coefficient of the erased configuration model.

	\begin{figure}[tb]
		\centering
		\begin{subfigure}{0.45\linewidth}
			\centering
			\includegraphics[width=0.45\textwidth]{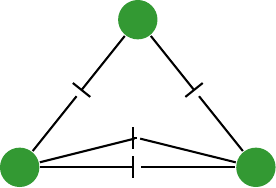}
			\caption{CM}
			\label{fig:triangCM}
		\end{subfigure}
		\begin{subfigure}{0.45\linewidth}
			\centering
			\includegraphics[width=0.45\textwidth]{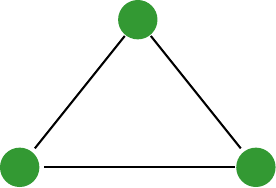}
			\caption{ECM}
			\label{fig:triangECM}
		\end{subfigure}
		\caption{From the edge perspective in the configuration model, these are two triangles. From the vertex perspective, there is only one triangle.}
		\label{fig:CMECM}
	\end{figure}
		
	
	
\section{Main results}\label{sec:main}

The next theorem presents our main result on the behavior of the local clustering coefficient in the erased configuration model. 
	\begin{theorem}\label{thm:ck}
		Let $G$ be an erased configuration model, where the degrees are an i.i.d.\ sample from a power-law distribution with exponent $\tau\in(2,3)$ as in~\eqref{D-tail} with $\tau\in(2,3)$. Define $A=-\Gamma(2-\tau)>0$ for $\tau\in(2,3)$ and let $\mu=\Exp{D}$.
		\begin{enumerate}[leftmargin=* , labelsep=1.5cm, 
			align=LeftAlignWithIndent, itemsep=-0.1cm]
			\item [\textup{(Range I.)}]  For $k\ll n^{(\tau-2)/(\tau-1)}$,
			\begin{equation}
				\frac{c(k)}{n^{2-\tau}\log(n)}\plim \frac{3-\tau}{\tau-1} \mu^{-\tau} C^2A.
			\end{equation}
			\item
			[\textup{(Range II.)}] When $an^{(\tau-2)/(\tau-1)}\leq k\ll \sqrt{n}$ for some $a>0$,
			\begin{equation}
				\frac{c(k)}{n^{2-\tau}\log(n/k^2)}\plim \mu^{-\tau} C^2 A .
			\end{equation}
			\item 
			[\textup{(Range III.)}] For $\sqrt{n}\ll k\ll d_{\max}$,
			\begin{equation}
				\frac{c(k)}{n^{5-2\tau}k^{2\tau-6}}\plim \mu^{3-2\tau}C^2 A^2.
			\end{equation}
		\end{enumerate} 
	\end{theorem}

Theorem~\ref{thm:ck} shows three different ranges for $k$ where $c(k)$ behaves differently, and is illustrated in Figure~\ref{fig:curve}. Let us explain why these three ranges occur. 
Range I contains small-degree vertices with $k\ll n^{(\tau-2)/(\tau-1)}$. In 
Section~\ref{sec:erdegrees} we show that in the configuration model these vertices are hardly involved in self-loops and multiple edges, and hence there is little difference between counting from an edge perspective or from a vertex perspective. It turns out that these vertices barely make triadic closures with hubs, which renders $c(k)$ independent of $k$ in Theorem~\ref{thm:ck}. 
Range II contains degrees that are neither small nor large with degrees 
 $ n^{(\tau-2)/(\tau-1)}\ll k\ll \sqrt{n}$. We can approximate the connection probability between vertices $i$ and $j$ with $1-\me^{-D_iD_j/\mu n}$, where $\mu=\expec[D]$. Therefore, a vertex of degree $k$ connects to vertices of degree at least $n/k$ with positive probability.
The vertices in Range II quite likely have multiple connections with vertices of degrees at least $n/k$. Thus, in this degree range, the single-edge constraint of the erased configuration model starts to play a role and causes the slow logarithmic decay of $c(k)$ in Theorem~\ref{thm:ck}. The vertices in this range turn out to be neighbors of hubs.  
Range III contains the large-degree vertices with $k\gg \sqrt{n}$.  Again we approximate the probability that vertices $i$ and $j$ are connected by $1-\me^{-D_iD_j/\mu n}$. This shows that vertices in Range III are likely to be connected to one another, possibly through multiple edges. The single-edge constraint on all connections between these core vertices causes the power-law decay of $c(k)$ in Theorem~\ref{thm:ck}.

\begin{figure}[h]
		\centering
		\begin{minipage}{0.45\linewidth}
			\centering
		\includegraphics[width=\textwidth]{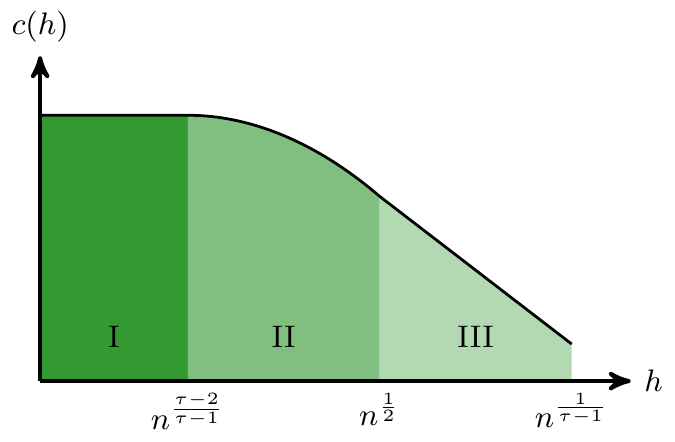}
		\caption{The three ranges of $c(k)$ defined in Theorem~\ref{thm:ck} on a log-log scale}
		\label{fig:curve}
	\end{minipage}
\hspace{1cm}
\begin{minipage}{0.45\linewidth}
	\centering
	\includegraphics[width=\textwidth]{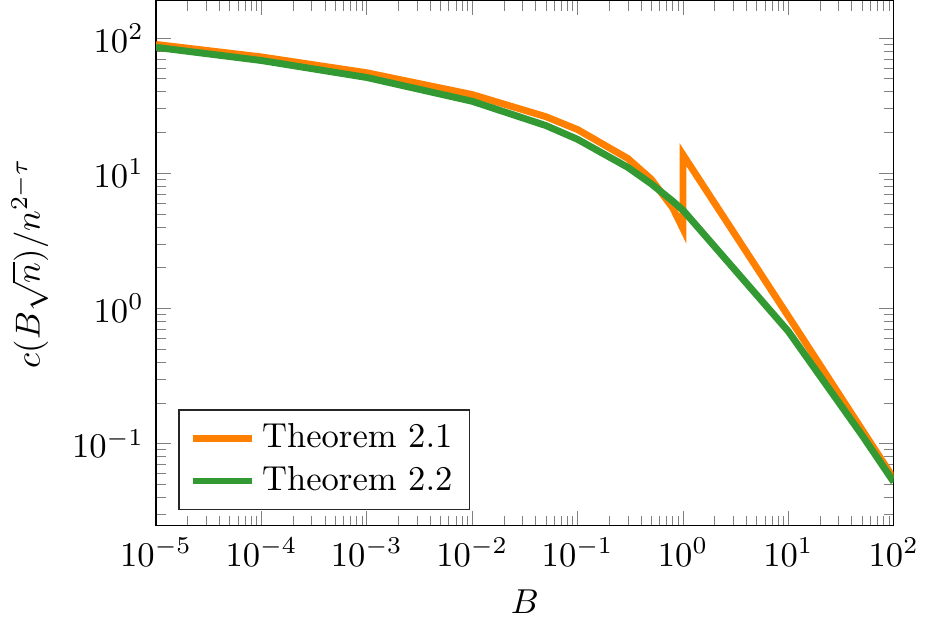}
	\caption{The normalized version of $c(k)$ for $k=B\sqrt{n}$ obtained from Theorem~\ref{thm:ck} and Theorem~\ref{thm:sqrt}.}
	\label{fig:Bsqrtn}
	\end{minipage}
\end{figure}

Observe that in Theorem ~\ref{thm:ck}  we write $\ll$ rather than $<$ for the values of $k$, because the behavior of $c(k)$  on the boundary between two different ranges may be different than the behavior inside the ranges. 
Since $k\mapsto c(k)$ is a function on a discrete domain, it is always continuous. However, we can extend the scaling limit of $k\mapsto c(k)$ to a continuous domain.
Theorem~\ref{thm:ck} then shows that the scaling limit of $k\mapsto c(k)$ is a smooth function inside the different ranges. Furthermore, filling in $k=an^{(\tau-1)/(\tau-2)}$ in Range II of Theorem~\ref{thm:ck} shows that $k\mapsto c(k)$ is also a smooth function on the boundary between Ranges I and II. However, the behavior of $k\mapsto c(k)$ on the boundary between Ranges II and III is not clear from Theorem~\ref{thm:ck}. We therefore prove the following result in Section~\ref{sec:proofthm}:

\begin{theorem}\label{thm:sqrt}
 For  $k=B\sqrt{n}$,
	\begin{equation}\label{eq:cksqrt}
	\frac{c(k)}{n^{2-\tau}}  \plim C^2\mu^{2-2\tau}B^{-2}\int_{0}^{\infty}\int_{0}^{\infty}(t_1t_2)^{-\tau}(1-\me^{-Bt_1})(1-\me^{-Bt_2})(1-\me^{-t_1t_2\mu})\dd t_1\dd t_2.
	\end{equation}
\end{theorem}

Figure~\ref{fig:Bsqrtn} compares $c(k)/n^{2-\tau}$ for $k=B\sqrt{n}$ using Theorem~\ref{thm:sqrt} and Theorem~\ref{thm:ck}. The line associated with Theorem~\ref{thm:ck} uses the result for Range II when $B<1$, and the result for Range III when $B>1$. We see that there seems to be a discontinuity between these two ranges. Figure~\ref{fig:Bsqrtn} suggests that the scaling limit of $k\mapsto c(k)$ is smooth around $k\approx \sqrt{n}$, because the lines are close for both small and large $B$-values. Theorem~\ref{theorem3} shows that indeed the scaling limit of $k\mapsto c(k)$ is smooth for $k$ of the order $\sqrt{n}$:

\begin{theorem}\label{theorem3}
The scaling limit of $k\mapsto c(k)$ is a smooth function.
\end{theorem}

\paragraph{Most likely configurations.}
	 The three different ranges in Theorem~\ref{thm:ck} result from a canonical trade-off caused by the power-law degree distribution. On the one hand, high-degree vertices participate in many triangles. In Section~\ref{sec:expmajor} we show that the probability that a triangle is present between vertices with degrees $k, D_u$ and $D_v$ can be approximated by 
	 \begin{equation}\label{eq:probtriangest}
		 \left(1-\me^{kD_u/\mu n}\right) \left(1-\me^{kD_v/\mu n}\right) \left(1-\me^{D_uD_v/\mu n}\right).
	 \end{equation}
	 The probability of this triangle thus increases with $D_u$ and $D_v$. On the other hand, in power-law distribution high degrees are rare. This creates a trade-off between the occurrence of triangles between $\{k, D_u,D_v \}$-triplets and the number of them. Surely, 
large degrees $D_u$ and $D_v$ make a triangle more likely, but larger degrees are less likely to occur. Since~\eqref{eq:probtriangest} increases only slowly in $D_u$ and $D_v$ as soon as $D_u,D_v\gg \mu n/k$ or when $D_uD_v\gg \mu n$, intuitively, triangles with $D_u,D_v\gg \mu n/k$ or with $D_uD_v\gg \mu n$ only marginally increase the number of triangles. 
	 In fact, we will show that most triangles with a vertex of degree $k$ contain two other vertices of very specific degrees, those degrees that can are aligned with the trade-off. The typical degrees of $D_u$ and $D_v$ in a triangle with a vertex of degree $k$ are given by $D_u,D_v\approx \mu n/k$ or by $D_uD_v\approx \mu n$. 
	 
	Let us now formalize this reasoning. Introduce
	\begin{equation}\label{eq:wkn}
	W_n^k(\varepsilon) = \begin{cases}
	(u,v): D_uD_v\in[\varepsilon,1/\varepsilon]\mu n & \text{for }k\ll n^{(\tau-2)/(\tau-1)},\\
	(u,v): D_uD_v\in[\varepsilon,1/\varepsilon]\mu n,  D_u,D_v<\mu n/(k\varepsilon)&\text{for } a n^{(\tau-2)/(\tau-1)}\leq k \ll\sqrt{n},\\
	(u,v): D_u,D_v\in [\varepsilon,1/\varepsilon]\mu n/k & \text{for } k\gg\sqrt{n},
	\end{cases}
	\end{equation}
	for some $a>0$. Denote the number of triangles between one vertex of degree $k$ and two other vertices $i,j$ with $(i,j)\in W_n^k(0)$ by $\triangle_k(W_n^k(\varepsilon))$. The next theorem shows that these types of triangles dominate all other triangles where one vertex has degree $k$:
	
	\begin{theorem}\label{thm:typetriang}
		Let $G$ be an erased configuration model where the degrees are an i.i.d.\ sample from a power-law distribution with exponent $\tau\in(2,3)$. Then, for $\varepsilon_n\to 0$ sufficiently slowly,
		\begin{equation}
			\frac{\triangle_k(W_n^k(\varepsilon_n))}{\triangle_k}\plim 1.
		\end{equation}
	\end{theorem}
	Figure~\ref{fig:major} illustrates the typical triangles containing a vertex of degree $k$ as given by Theorem~\ref{thm:typetriang}. When $k$ is small ($k$ in Range I or II), a typical triangle containing a vertex of degree $k$ is a triangle with vertices $u$ and $v$ such that $D_uD_v=\Theta(n)$ as shown in in Figure~\ref{fig:contsmall}. Then, the probability that an edge between $u$ and $v$ exists is asymptotically positive and non-trivial. Since $k$ is small, the probability that an edge exists between a vertex of degree $k$ and $u$ or $v$ is small. On the other hand, when $k$ is larger (in Range III), a typical triangle containing a vertex of degree $k$ is with vertices $u$ and $v$ such that $D_u=\Theta(n/k)$ and $D_v=\Theta(n/k)$. Then, the probability that an edge exists between $k$ and $D_u$ or $k$ and $D_v$  is asymptotically positive whereas the probability that an edge exists between vertices $u$ and $v$ vanishes. Figure~\ref{fig:contlarge} shows this typical triangle. 
	\begin{figure}[tb]
		\centering
		\begin{subfigure}{0.45\linewidth}
			\centering
			\includegraphics[width=0.7\textwidth]{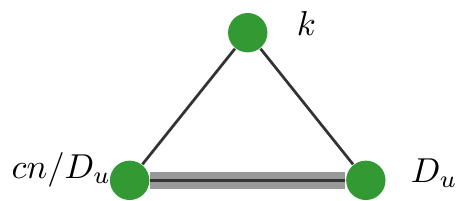}
			\caption{$k<\sqrt{n}$}
			\label{fig:contsmall}
		\end{subfigure}
		\begin{subfigure}{0.45\linewidth}
			\centering
			\includegraphics[width=0.7\textwidth]{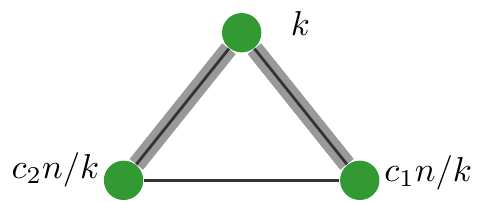}
			\caption{$k>\sqrt{n}$}
			\label{fig:contlarge}
		\end{subfigure}
		\caption{The major contributions in the different ranges for $k$. The highlighted edges are present with asymptotically positive probability.}
		\label{fig:major}
	\end{figure}
	
	Figure~\ref{fig:contplot} shows the typical size of the degrees of other vertices in a triangle with a vertex of degree $k=n^\beta$. We see that when $\beta<(\tau-2)/(\tau-1)$ (so that $k$ is in Range I), the typical other degrees are independent of the exact value of $k$. This shows why $c(k)$ is independent of $k$ in Range I in Theorem~\ref{thm:ck}. When $(\tau-2)/(\tau-1)<\beta<\tfrac12$, we see that the range of possible degrees for vertices $u$ and $v$ decreases when $k$ gets larger. Still, the range of possible degrees for $D_u$ and $D_v$ is quite wide. This explains the mild dependence of $c(k)$ on $k$ in Theorem~\ref{thm:ck} in Range II. When $\beta>\tfrac12$, $k$ is in Range III. Then the typical values of $D_u$ and $D_v$ are considerably different from those in the previous regime. 
	The values that $D_u$ and $D_v$ can take depend heavily on the value of $k$. This explains the dependence of $c(k)$ on $k$ in Range III.
	\begin{figure}[tb]
		\centering
		\includegraphics[width=0.4\textwidth]{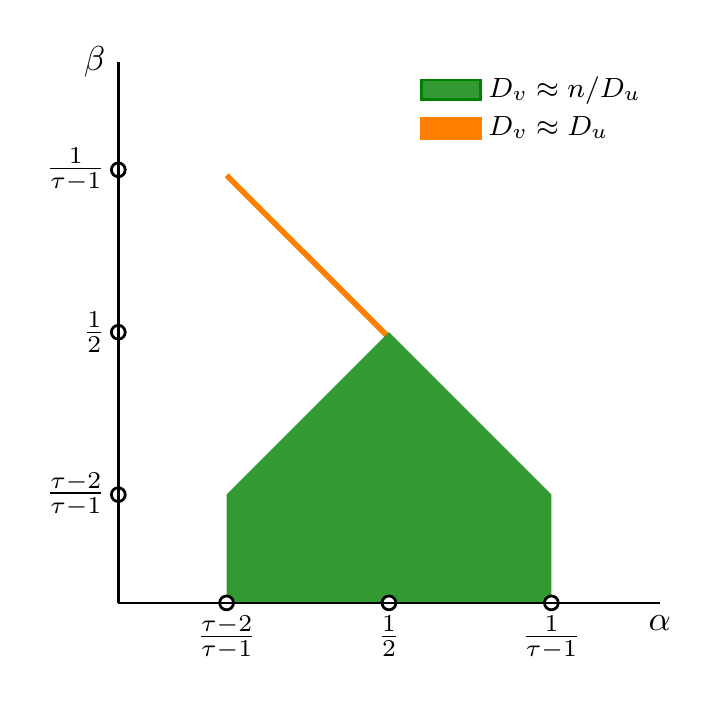}
		\caption{Visualization of the contributing degrees when $k=n^{\beta}$ and $D_u=n^{\alpha}$. The colored area shows the values of $\alpha$ that contribute to $c(n^\beta)$.}
		\label{fig:contplot}
	\end{figure}
	
	\paragraph{Global and local clustering.}
	The global clustering coefficient divides the total number of triangles by the total number of pairs of neighbors of all vertices. 
	In~\cite{hofstad2017d}, we have shown that the total number of triangles in the configuration model from a vertex perspective is determined by vertices of degree proportional to $\sqrt{n}$. Thus, only triangles between vertices on the border between Ranges II and III contribute to the global clustering coefficient. The local clustering coefficient counts all triangles where one vertex has degree $k$ and provides a more complete picture of clustering from a vertex perspective, since it covers more types of triangles. 
%
	
	\paragraph{Hidden-variable models.}
	Our results for clustering in the erased configuration model agree with recent results for the hidden-variable model~\cite{stegehuis2017}. In the hidden-variable model, every vertex is equipped with a hidden variable $w_i$, where the hidden variables are sampled from a power-law distribution. Then, vertices $i$ and $j$ are connected with probability $\min(w_iw_j/n,1)$~\cite{chung2002,boguna2003}. In the erased configuration model, we will use that the probability that a vertex with degree $D_i$ is connected to a vertex with degree $D_j$ can be approximated by
	\begin{equation}
	1-\me^{-D_iD_j/\mu n},
	\end{equation}
which behaves similarly as $\min(D_iD_j/n,1)$. Thus, the connection probabilities in the erased configuration model can be interpreted as the connection probabilities in the hidden-variable model, where the sampled degrees can be interpreted as the hidden variables.
	The major difference is that connections in the hidden-variable model are {\em independent} once the hidden variables are sampled, whereas connections in the erased configuration model are correlated once the degrees are sampled. Indeed, in the erased configuration model we know that a vertex with degree $D_i$ has at most $D_i$ other vertices as a neighbor, so that the connections from vertex $i$ to other vertices are correlated. Still, our results show that these correlations are small enough for the results for $c(k)$ to be similar to the results for $c(k)$ in the hidden variable model.
	
	\subsection{Overview of the proof}
To prove Theorem~\ref{thm:ck}, we show that there is a major contributing regime for $c(k)$, which characterizes the degrees of the other two vertices in a typical triangle with a vertex of degree $k$. We write this major contributing regime as $W_n^k(\varepsilon)$ defined in~\eqref{eq:wkn}. 
The number of triangles adjacent to a vertex of degree $k$ is dominated by triangles between the vertex of degree $k$ and other vertices with degrees in a specific regime, depending on $k$. All three ranges of $k$ have a different spectrum of degrees that contribute to the number of triangles. 
We write
	\begin{equation}
	c(k)= c(k,W_n^k(\varepsilon))+c(k,\bar{W}_n^k(\varepsilon)),
	\end{equation}
where $c(k,\bar{W}_n(\varepsilon))$ denotes the contribution to $c(k)$ from triangles where the other two vertices $u$ and $v$ satisfy $(D_u,D_v)\notin W_n^k(\varepsilon)$.
Furthermore, we will write the order of magnitude of the value of $c(k)$ as $f(k,n)$. Theorem~\ref{thm:ck} states that this order should be
	\begin{equation}\label{eq:fkn}
	f(k,n) = 
	\begin{cases}
	n^{2-\tau}\log(n)  & 	 \text{for }k\ll n^{(\tau-2)/(\tau-1)},\\
	n^{2-\tau}\log(n/k^2) & 	 \text{for } an^{(\tau-2)/(\tau-1)}\leq k\ll \sqrt{n},\\
	n^{5-2\tau}k^{2\tau-6} & \text{for } k\gg\sqrt{n},
	\end{cases}
	\end{equation}	
	for some $a>0$.
The proof of Theorem~\ref{thm:ck} is largely built on the following two propositions:

\begin{proposition}[Main contribution]\label{prop:major}
	\begin{equation}
	\frac{c(k,W_n^k(\varepsilon))}{f(n,k)} \plim \begin{cases}
	C^2\int_{\varepsilon}^{1/\varepsilon}t^{1-\tau}(1-\me^{-\tau})\dd t & k\ll \sqrt{n},\\
	C^2\left(\int_{\varepsilon}^{1/\varepsilon}t^{1-\tau}(1-\me^{-\tau})\dd t \right)^2& k\gg \sqrt{n}.\\
	\end{cases}
	\end{equation}
\end{proposition}

\begin{proposition}[Minor contributions]\label{prop:minor}
	There exists $\kappa>0$ such that for all ranges
	\begin{equation}
	\limsup_{n\to\infty}\frac{\Expn{c(k,\bar{W}_n^k(\varepsilon))}}{f(n,k)} \plim \bigOp{\varepsilon^\kappa}.
	\end{equation}
\end{proposition}
	We now show how these propositions prove Theorem~\ref{thm:ck}. Applying Proposition~\ref{prop:minor} together with the Markov inequality yields 
	\begin{equation}
	\Prob{c(k,\bar{W}_n^k(\varepsilon))>Kf(k,n)\varepsilon^\kappa} = \bigO{K^{-1}}.
	\end{equation}
	Therefore,
	\begin{equation}
	c(k)=c(k,W_n^k(\varepsilon))+O_\prob\left(f(k,n)\varepsilon^\kappa\right).
	\end{equation}
	Taking the limit of $\varepsilon\to 0$ then already proves Theorem~\ref{thm:typetriang}. 
	To prove Theorems~\ref{thm:ck} and~\ref{thm:sqrt} we use Proposition~\ref{prop:major}, which shows that
	\begin{equation}
	\frac{c(k,W_n^k(\varepsilon))}{f(k,n)}\plim 
	\begin{cases}
	C^2\int_\varepsilon^{1/\varepsilon}t^{1-\tau}(1-\me^{-t})\dd t+O(\varepsilon^\kappa) &k \ll\sqrt{n},\\
	C^2\left(\int_\varepsilon^{1/\varepsilon}t^{1-\tau}(1-\me^{-t})\dd t\right)^2+O(\varepsilon^\kappa) & k\gg\sqrt{n}.
	\end{cases}
	\end{equation}
	We take the limit of $\varepsilon\to 0 $ and use that
	\begin{equation}
	\begin{aligned}[b]
	\int_{0}^{\infty}x^{1-\tau}(1-\me^{-x})\dd x & = \int_{0}^{\infty}\int_{0}^{x}x^{1-\tau}\me^{-y}\dd y \dd x =  \int_{0}^{\infty}\int_{y}^{\infty}x^{1-\tau}\me^{-y}\dd x \dd y\\
	& = -\frac{1}{2-\tau}\int_{0}^{\infty}y^{2-\tau}\me^{-y}\dd y = -\frac{\Gamma(3-\tau)}{2-\tau}=-\Gamma(2-\tau)=:A,
	\end{aligned}
	\end{equation}
	which proves Theorem~\ref{thm:ck}.
	
	The rest of the paper will be devoted to proving Propositions~\ref{prop:major} and~\ref{prop:minor}. We prove Proposition~\ref{prop:major} using a  second moment method. We can compute the expected value of $c(k)$ conditioned on the degrees as
	\begin{equation}\label{eq:condexeq}
		\Expn{c(k)}=\frac{2\Expn{
		\sum_{w:\Der_w=k}	\triangle(w)}}{N_kk(k-1)},
	\end{equation}
	where $\triangle(w)$ denotes the number of triangles containing vertex $w$ and $\expec_n$ denotes the conditional expectation given the degrees. 
	Let $X_{ij}$ denote the number of edges between vertex $i$ and $j$ in the configuration model, and $\hat{X}_{ij}$ the number of edges between $i$ and $j$ in the corresponding erased configuration model, so that $\hat{X}_{ij}\in\{0,1\}$. 
	Now,
	\begin{equation}\label{eq:triangeq}
		\Expn{\triangle(w)\mid \Der_w=k}=\tfrac{1}{2}\sum_{u,v\neq w} \prob_n(\hat{X}_{wu}= \hat{X}_{wv}=\hat{ X}_{uv}= 1\mid \Der_w=k).
	\end{equation}
	Thus, to find the expected number of triangles, we need to compute the probability that a triangle between vertices $u$, $v$ and $w$ exists, which we will do in Section~\ref{sec:expmajor}. After that, we show that this expectation converges to a constant when taking the randomness of the degrees into account, and that the variance conditioned on the degrees is small in Section~\ref{sec:var}. 
	Then, we prove Proposition~\ref{prop:minor} in Section~\ref{sec:minor} using a first moment method. We start in Section \ref{sec:prelim} to state some preliminaries.

	\section{Preliminaries}\label{sec:prelim}
	We now introduce some lemmas that we will use frequently while proving Propositions~\ref{prop:major} and~\ref{prop:minor}. We let $\prob_n$ denote the conditional probability given $\bfDit$, and $\expec_n$ the corresponding expectation.
	 Furthermore, let $\mathcal{D}_u$ denote a uniformly chosen vertex from the degree sequence and let $L_n=\sum_{i\in[n]}D_i$ denote the sum of the degrees. 
	
	\subsection{Conditioning on the degrees}\label{sec:conddeg}
	In the proof of Proposition~\ref{prop:major} we will first condition on the degree sequence. We compute the clustering coefficient conditional on the degree sequence, and after that we show that this converges to the correct value when taking the random degrees into account. We will use the following lemma several times:
	\begin{lemma}\label{lem:conddeg}
		Let $G$ be an erased configuration model where the degrees are an i.i.d.\ sample from a random variable $D$. Then,
		\begin{align}
			\Probn{\mathcal{D}_u\in[a,b]} & = \bigOp{\Prob{D\in[a,b]}}\\
			\Expn{f(\mathcal{D}_u)}&=\bigOp{\Exp{f(D)}}.
		\end{align}
	\end{lemma}
\begin{proof}
	By using the Markov inequality, we obtain for $M>0$
	\begin{equation}
	\begin{aligned}[b]
		\Prob{\Probn{\mathcal{D}_u\in[a,b]}\geq M \Prob{D\in[a,b]}}\leq \frac{\Exp{\Probn{\mathcal{D}_u\in[a,b]}}}{M\Prob{D\in[a,b]}}=\frac{1}{M},
	\end{aligned}
	\end{equation}
	and the second claim can be proven in a very similar way.
\end{proof}

In the proof of Theorem~\ref{thm:ck} we will often estimate moments of $D$, conditional on the degrees. The following lemma shows how to bound these moments, and is a direct consequence of the Stable Law Central Limit Theorem:
\begin{lemma}\label{lem:expind}
	Let $\mathcal{D}_u$ be a uniformly chosen vertex from the degree sequence, where the degrees are an i.i.d.\ sample from a power-law distribution with exponent $\tau\in(2,3)$. Then, for $\alpha>\tau-1$,
	\begin{equation}
		\Expn{\mathcal{D}_u^\alpha}=\bigOp{n^{\alpha/(\tau-1)-1}}.
	\end{equation}
\end{lemma}
\begin{proof}
	We have
	\begin{equation}
		\Expn{\mathcal{D}_u^\alpha}=\frac{1}{n}\sum_{i=1}^{n}D_i^\alpha.
	\end{equation}
	Since the $D_i$ are an i.i.d.\ sample from a power-law distribution with exponent $\tau$, $D_i^\alpha$ are distributed as i.i.d.\ samples from a power-law with exponent $(1-\tau)/\alpha+1<1$. 
	Then, by the Stable law Central Limit Theorem (see for example~\cite[Theorem 4.5.1]{whitt2006}),
	\begin{equation}
		\sum_{i=1}^{n}D_i^\alpha = \bigOp{n^{\frac{\alpha}{\tau-1}}},
	\end{equation}
	which proves the lemma.
\end{proof}

We also need to relate $L_n$ and its expected value $\mu n$. Define the event
	\begin{equation}\label{eq:Jn}
	J_n = \left\{ \abs{L_n-\mu n}\leq n^{1/(\tau-1)}\right\}.
	\end{equation}
By~\cite{Hoorn2015}, $\Prob{J_n}\to 1$ as $n\to\infty$. When we condition on the degree sequence, we will assume that the event $J_n$ takes place.

	\subsection{Erased and non-erased degrees}\label{sec:erdegrees}	
	The degree sequence of the erased configuration model may differ from the original degree sequence of the original configuration model. We now show that this difference is small with high probability. 
By~\cite[Eq A(9)]{britton2006}, the probability that a half-edge incident to a vertex of degree $o(n)$ is removed is $o(1)$. Therefore,
	\begin{equation}
	\Der_i=D_i(1+\op(1))
	\end{equation}
as long as $D_i=o(n)$. Since the maximal degree in the configuration model with i.i.d.\ degrees is $O_\prob(n^{1/(\tau-1)})$, $D_i=\op(n)$ uniformly in $i$. 
Thus, in many proofs, we will exchange $D_i$ and $\Der_i$ when needed.

	\section{Second moment method on main contribution $W_n^k(\varepsilon)$}\label{sec:prop1}
	We now focus on the triangles that give the main contribution. First, we condition on the degree sequence and compute the expected number of triangles in the main contributing regime. Then, we show that this expectation converges to a constant when taking the i.i.d.\ degrees into account. After that, we show that the variance of the number of triangles in the main contributing regime is small, and we prove Proposition~\ref{prop:major}.
	
	\subsection{Conditional expectation inside $W_n^k(\varepsilon)$}\label{sec:expmajor}
	
	In this section, we compute the expectation of the number of triangles in the major contributing ranges of~\ref{eq:wkn} when we condition on the degree sequence.
	 We define
	\begin{equation}\label{eq:gn}
	g_n(D_u,D_v,D_w):=(1-\me^{-D_uD_v/L_n})(1-\me^{-D_uD_w/L_n})(1-\me^{-D_vD_w/L_n}).
	\end{equation}
	Then, the following lemma shows that the expectation of $c(k)$ conditioned on the degrees is the sum of $g_n(D_u,D_v,D_w)$ over all degrees in the major contributing regime:
	
	\begin{lemma}\label{lem:expk}
		On the event $J_n$, 
		\begin{equation}\label{eq:exk}
		\Expn{c(k,W_n^k(\varepsilon))}=\frac{ \sum_{(u,v)\in W_n^k(\varepsilon)}g_n(k,D_u,D_v)}{k(k-1)} (1+\op(1)).
		\end{equation}
	\end{lemma}
	\begin{proof}
	We write the probability that a specific triangle exists as
	\begin{equation}
		\begin{aligned}[b]
			\Probn{\triangle_{u,v,w}=1} & =1- \Probn{{X}_{uw}=0}-\Probn{{X}_{vw}=0}-\Probn{{X}_{uv}=0}      +\Probn{{X}_{uw}=X_{vw}=0}                             \\
			&\quad + \Probn{{X}_{uv}=X_{vw}=0}+ \Probn{{X}_{uv}=X_{uw}=0}- \Probn{X_{uv}={X}_{uw}=X_{vw}=0}.
		\end{aligned}
	\end{equation}
	In the major contributing ranges, $D_u,D_v,D_w=\Omega(n^{(\tau-2)/(\tau-1)})$ and $D_u,D_v,D_w=O_{\sss \prob}(n^{1/(\tau-1)})$, and the product of the degrees is $O(n)$. By~\cite[Lemma 3.1]{hofstad2017d}
	\begin{equation}
		\Probn{X_{uv}=X_{vw}=0}=\me^{-D_uD_v/L_n}\me^{-D_vD_w/L_n}(1+\op(n^{-(\tau-2)/(\tau-1)}))
	\end{equation}
	and
	\begin{equation}
	\Probn{X_{uv}=X_{vw}=X_{uw}=0}=\me^{-D_uD_v/L_n}\me^{-D_vD_w/L_n}\me^{-D_uD_w/L_n}(1+\op(n^{-(\tau-2)/(\tau-1)})).
	\end{equation}
	Therefore,
	\begin{equation}\label{eq:probtriang}
	\begin{aligned}[b]
	\Probn{\triangle_{u,v,w}=1} 
	& =(1+\op(1))\left(1-\me^{-D_uD_v/L_n}\right)\left(1-\me^{-D_uD_w/L_n}\right)\left(1-\me^{-D_vD_w/L_n}\right)\\
	& = (1+\op(1))g_n(D_u,D_v,D_w),
	\end{aligned}
	\end{equation}
where we have used that for $D_uD_v=O(n)$
	\begin{equation}
		1-\me^{-D_uD_v/L_n}(1+\op(n^{-(\tau-2)/(\tau-1)}))=(1-\me^{-D_uD_v/L_n})(1+\op(1)).
	\end{equation}
We can use Lemma~\ref{lem:conddeg} to show that, given $\Der_w=k$,
	\begin{equation}\label{eq:gner}
	g_n(D_w,D_u,D_v)=g_n(k,D_u,D_v)(1+\op(1)).
	\end{equation}
Then,~\eqref{eq:condexeq} and~\eqref{eq:triangeq} show that
	\begin{equation}
		\Expn{c(k,W_n^k(\varepsilon))}=\frac{\frac{1}{2 N_k}\sum_{w:\Der_w=k}\sum_{(u,v)\in W_n^k(\varepsilon)}\Probn{\triangle_{u,v,w}=1}}{k(k-1)/2}.
\end{equation}
	Thus, we obtain 
	\begin{equation}
	\begin{aligned}[b]
	\Expn{c(k,W_n^k(\varepsilon))}&=\frac{\sum_{w:\Der_w=k}\sum_{(u,v)\in W_n^k(\varepsilon)}g_n(D_w,D_u,D_v)}{N_kk(k-1)}(1+\op(1))\\
	&=\frac{\sum_{(u,v)\in W_n^k(\varepsilon)}g_n(k,D_u,D_v)}{k(k-1)}(1+\op(1)),
	\end{aligned}
	\end{equation}
which proves the lemma.
	\end{proof}

	\subsection{Analysis of asymptotic formula}\label{sec:conv}
	In the previous section, we have shown that the expected value of $c(k)$ in the major contributing regime is the sum of a function $g_n(k,D_u,D_v)$ over all vertices $u$ and $v$ with degrees in the major contributing regime if we condition on the degrees, that is
	\begin{equation}\label{eq:ckexp}
		\Expn{c(k,W_n^k(\varepsilon))}=\frac{2}{k(k-1)}\sum_{u,v: (D_u,D_v)\in W_n^k(\varepsilon)}(1-\me^{-kD_v/L_n})(1-\me^{-kD_u/L_n})(1-\me^{-D_uD_v/L_n}).
	\end{equation} 
	This expected value does not yet take into account that the degrees are sampled i.i.d.\ from a power-law distribution. In this section, we will prove that this expected value converges to a constant when we take the randomness of the degrees into account. We will make use of the following lemmas:	
	
	\begin{lemma}\label{lem:convmeas2}
		Let $A\subset \mathbb{R}^2$ be a bounded set and $f(t_1,t_2)$ be a bounded, continuous function on $A$. Let $\Mn$ be a random measure such that for all $S\subseteq A$, $\Mn(S)\plim \lambda(S)=\int_S\dd \lambda(t_1,t_2)$ for some deterministic measure $\lambda$. Then, 
		\begin{equation}
			\int_A f(t_1,t_2)\dd \Mn(t_1,t_2)\plim \int_A f(t_1,t_2)\dd \lambda(t_1,t_2).
		\end{equation}
	\end{lemma}
	\begin{proof}
		Fix $\eta>0$. 
		Since $f$ is bounded and continuous on $A$, for any $\varepsilon>0$, we can find $m<\infty $, disjoint sets $(B_i)_{i\in[m]}$ and constants $(b_i)_{i\in[m]}$ such that $\cup B_i = A$ and
		\begin{equation}
			\abs{f(t_1,t_2)-\sum_{i=1}^{m}b_i\ind{(t_1,t_1)\in B_i}}<\varepsilon,
		\end{equation}
		for all $(t_1,t_2)\in A$. 
		Because $\Mn(B_i)\plim \lambda(B_i)$ for all $i$,
		\begin{equation}
			\lim_{n\to\infty}\Prob{\abs{\Mn(B_i)-\lambda(B_i)}>\eta/m}=0.
		\end{equation}
		Then,
		\begin{equation}
			\begin{aligned}[b]
			\abs{\int_A f(t_1,t_2)\dd \Mn(t_1,t_2)-\int_A f(t_1,t_2)\dd \lambda(t_1,t_2)}
			& \leq \abs{\int_A f(t_1,t_2)-\sum_{i=1}^mb_i \ind{(t_1,t_2)\in B_i}\dd \Mn(t_1,t_2)}\\
				& \quad+ \abs{\int_A f(t_1,t_2)-\sum_{i=1}^mb_i \ind{(t_1,t_2)\in B_i}\dd \lambda(t_1,t_2)}\\
				&\quad  +\abs{\sum_{i=1}^mb_i(\Mn(B_i)-\lambda(B_i))}\\
				& \leq  \varepsilon \Mn(A)+\varepsilon \lambda(A)+\op(\eta).
			\end{aligned}
		\end{equation}
		Now choosing $\varepsilon<\eta/\lambda(A)$ proves the lemma.
	\end{proof}
	
The following lemma is a straightforward one-dimensional version of Lemma~\ref{lem:convmeas2}.
\begin{lemma}\label{lem:convmeas}
	Let $\Mn[a,b]$ be a random measure such that for all $0<a<b$, $\Mn[a,b]\plim \lambda[a,b]=\int_a^b\dd \lambda(t)$ for some deterministic measure $\lambda$. Let $f(t)$ be a bounded, continuous function on $[\varepsilon,1/\varepsilon]$. Then, 
	\begin{equation}
	\int_{\varepsilon}^{1/\varepsilon}f(t)\dd \Mn(t)\plim \int_{\varepsilon}^{1/\varepsilon}f(t)\dd \lambda(t).
	\end{equation}
\end{lemma}
\begin{proof}
	This proof follows the same lines as the proof of Lemma~\ref{lem:convmeas2}.
\end{proof}
	
	Using these lemmas we investigate the convergence of the expectation of $c(k)$ conditioned on the degrees. We treat the three ranges separately, but the proofs follow the same structure. First, we define a random measure $\Mn$ that counts the normalized number of vertices with degrees in the major contributing regime. We then show that this measure converges to a deterministic measure $\lambda$, using that the degrees are i.i.d.\ samples of a power-law distribution. We then write the conditional expectation of the previous section as an integral over measure $\Mn$. Then, we can use Lemmas~\ref{lem:convmeas} or~\ref{lem:convmeas2} to show that this converges to a deterministic integral. 
	\medskip
	
	First, we consider the case where $k$ is in Range I.
	\begin{lemma}\label{lem:convksmall}
		\textup{(Range I)} For $k=o(n^{(\tau-2)/(\tau-1)})$,
		\begin{equation}
		\frac{\Expn{c(k,W_n^k(\varepsilon))}}{n^{2-\tau}\log(n)}\plim \mu^{-\tau} C^2\frac{3-\tau}{\tau-1}\int_{\varepsilon}^{1/\varepsilon}t^{1-\tau}(1-\me^{-t})\dd t .
		\end{equation}
		\end{lemma}
	\begin{proof}
		Since the degrees are i.i.d.\ samples from a power-law distribution, $D_u=O_\prob(n^{1/(\tau-1)})$ uniformly in $u\in[n]$. Thus, when $k=o(n^{(\tau-2)/(\tau-1)})$, $kD_u=\op(n)$ uniformly in $u\in[n]$. Therefore, we can Taylor expand the first two exponentials in~\eqref{eq:ckexp}, using that $1-\e^{-x}=x+O(x^2)$. By Lemma~\ref{lem:expk}, this leads to
	\begin{equation}
		\Expn{c(k,W_n^k(\varepsilon))}=(1+\op(1)) \frac{k^2}{k(k-1)} \sum_{u,v: (D_u,D_v)\in W_n^k(\varepsilon)} \frac{D_uD_v(1-\e^{-D_uD_v/L_n})}{L_n^2}.
	\end{equation}	
Furthermore, since $D_u=O_\prob(n^{1/(\tau-1)})$ while also $D_uD_v=\Theta(n)$ in the major contributing regime, we can add the indicator that $K_1n^{(\tau-2)/(\tau-1)}<D_u<K_2n^{1/(\tau-1)}$ for $0<K_1,K_2<\infty$. 
We then define the random measure 
	\begin{equation}\label{eq:Mn}
	\Mn[a,b] = \frac{(\mu n)^{\tau-1}}{\log(n)n^2}\sum_{u,v\in[n]}\ind{D_uD_v\in n\mu [a,b], K_1n^{(\tau-2)/(\tau-1)}<D_u<K_2n^{1/(\tau-1)}}.
	\end{equation} 
We can write the expected value of this measure as
	\begin{equation}
	\begin{aligned}[b]
	\Exp{\Mn[a,b]} & = \frac{(\mu n)^{\tau-1}}{\log(n)n^2}\Exp{\abs{\left\{i,j: D_iD_j\in[a,b]\mu n, D_i\in [K_1n^{(\tau-2)/(\tau-1)},K_2n^{\frac{1}{\tau-1}}]\right\}}}\\
	& = \frac{(\mu n)^{\tau-1}}{\log(n)}\Prob{D_1D_2\in [a,b]\mu n,  D_1\in [K_1n^{(\tau-2)/(\tau-1)},K_2n^{\frac{1}{\tau-1}}]}                                        \\
	& = \frac{(\mu n)^{\tau-1}}{\log(n)}\int_{K_1n^{(\tau-2)/(\tau-1)}}^{K_2n^{\frac{1}{\tau-1}}}\int_{a\mu n/x}^{b\mu n/x}c^2(xy)^{-\tau}\dd y \dd x \\
	& = c^2\frac{(\mu n)^{\tau-1}}{\log(n)}\int_{K_1n^{(\tau-2)/(\tau-1)}}^{K_2n^{\frac{1}{\tau-1}}}\frac 1x \dd x \int_{a\mu n}^{b\mu n}u^{-\tau}\dd u \\
	& =C^2\int_{a}^{b}t^{-\tau}\dd t \left(\frac{3-\tau}{\tau-1}+\frac{\log(K_2/K_1)}{\log(n)}\right).                                                                                                                              
	\end{aligned}
	\end{equation}
Thus,
	\begin{equation}
		\lim_{n\to\infty}\Exp{\Mn[a,b]}= C^2\frac{3-\tau}{\tau-1}\int_{a}^{b}t^{-\tau}\dd t =:\lambda [a,b].
	\end{equation}
Furthermore, the variance of this measure can be written as
	\begin{equation}
	\begin{aligned}[b]
	\Var{\Mn[a,b]}& = \frac{(\mu n)^{2\tau-6}}{\log^2(n)}\sum_{u,v,w,z}\big(\Prob{D_uD_v,D_wD_z\in \mu n[a,b],D_u,D_w\in [K_1n^{(\tau-2)/(\tau-1)},K_2n^{\frac{1}{\tau-1}}]}\\
	&\quad -\Prob{D_uD_v\in \mu n[a,b], D_u\in [K_1n^{(\tau-2)/(\tau-1)},K_2n^{\frac{1}{\tau-1}}]}\\
	&\quad \times \Prob{D_wD_z\in \mu n[a,b],D_w\in [K_1n^{(\tau-2)/(\tau-1)},K_2n^{\frac{1}{\tau-1}}]}\big).
	\end{aligned}
	\end{equation}
Since the degrees are sampled i.i.d.\ from a power-law distribution, the contribution to the variance for $|\{u,v,w,z\}|=4$ is zero. 
The contribution from $|\{u,v,w,z\}|=3$ can be bounded as
	\begin{equation}
	\begin{aligned}[b]
	\frac{(\mu n)^{2\tau-6}}{\log^2(n)}\sum_{u,v,w}\Prob{D_uD_v,D_uD_w\in \mu n[a,b]}& = \frac{\mu^{2\tau-6}n^{2\tau-3}}{\log^2(n)}\Prob{{D}_1{D}_2,{D}_1{D}_3\in \mu n[a,b]}\\
	&  = \frac{\mu^{2\tau-6}n^{2\tau-3}}{\log^2(n)}\int_{1}^{\infty}cx^{-\tau}\left(\int_{an/x}^{bn/x}cy^{-\tau}\dd y\right)^2\dd x\\
	& \leq K\frac{n^{-1}}{\log^2(n)},
	\end{aligned}
	\end{equation}
for some constant $K$. Similarly, the contribution for $u=z$, $v=w$ can be bounded as
	\begin{equation}
	\begin{aligned}[b]
	\frac{(\mu n)^{2\tau-6}}{\log^2(n)}\sum_{u,v}\Prob{D_uD_v\in \mu n[a,b]}& = \frac{\mu^{2\tau-6}n^{2\tau-4}}{\log^2(n)}\Prob{{D}_1{D}_2\in \mu n[a,b]}\\
	&  \leq K\frac{n^{2\tau-4}}{\log^2(n)}n^{1-\tau}\log(n)= C\frac{n^{\tau-3}}{\log(n)},
	\end{aligned}
	\end{equation}
for some constant $K$. Thus, $\Var{\Mn[a,b]}=o_\prob(1)$. Therefore, a second moment method yields that for every $a,b>0$, 
	\begin{equation}\label{eq:convbin}
		\Mn[a,b]\plim \lambda[a,b].
	\end{equation}
	Then, 
	\begin{equation}
	\begin{aligned}[b]
	\sum_{u,v: (D_u,D_v)\in W_n^k(\varepsilon)} \frac{D_uD_v(1-\e^{-D_uD_v/L_n})}{L_n^2} & =\mu^{1-\tau}n^{3-\tau}\log(n)\int_{\varepsilon}^{1/\varepsilon}\frac{t}{L_n}(1-\me^{-t})\dd \Mn(t) \\
	& =\mu^{-\tau}n^{2-\tau}\log(n)\int_{\varepsilon}^{1/\varepsilon}t(1-\me^{-t})\dd \Mn(t)(1+o_\prob(1)).          
	\end{aligned}
	\end{equation}
By Lemma~\ref{lem:convmeas},
	\begin{equation}
	\begin{aligned}[b]
	\int_{\varepsilon}^{1/\varepsilon}t(1-\me^{-t})\dd \Mn(t) & \plim \int_{\varepsilon}^{1/\varepsilon}t(1-\me^{-t})\dd \lambda(t)\\
	&  =  C^2\frac{3-\tau}{\tau-1} \int_{\varepsilon}^{1/\varepsilon}t^{1-\tau}(1-\me^{-t})\dd t.
	\end{aligned}
	\end{equation}
If we first let $n\to\infty$, and then $K_1\to 0$ and $K_2\to\infty$, then we obtain
	\begin{equation}
	\frac{\Expn{c(k),W_n^k(\varepsilon)}}{n^{2-\tau}\log(n)}\plim C^2\mu^{-\tau}\frac{3-\tau}{\tau-1} \int_{\varepsilon}^{1/\varepsilon}t^{1-\tau}(1-\me^{-t})\dd t.
	\end{equation}
	\end{proof}
	
	
	\begin{lemma}\label{lem:convkmiddle}
		\textup{(Range II)} When $a n^{(\tau-2)/(\tau-1)}\leq k\ll\sqrt{n}$ for some $a>0$,
		\begin{equation}
			\frac{\Expn{c(k,W_n^k(\varepsilon))}}{n^{2-\tau}\log(n/k^2)}\plim C^2\mu^{-\tau}\int_{\varepsilon}^{1/\varepsilon}t^{1-\tau}(1-\me^{-t})\dd t.
		\end{equation}
	\end{lemma}
	\begin{proof}
		We split the major contributing regime into three parts, depending on the values of $D_u$ and $D_v$, as visualized in Figure~\ref{fig:contpart}. 
		\begin{figure}[tb]
			\centering
			\includegraphics[width=0.5\textwidth]{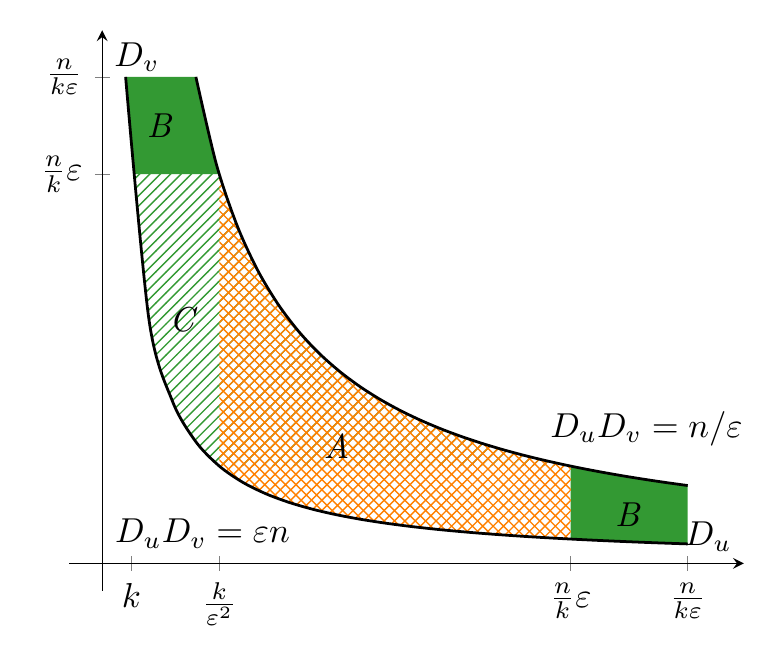}
			\caption{Contributing regime for $n^{(\tau-2)/(\tau-1)}<k<\sqrt{n}$}
			\label{fig:contpart}
		\end{figure}
				We denote the contribution to the clustering coefficient where $D_u\in[k/\varepsilon^2,\varepsilon n/k]$ (area A of Figure~\ref{fig:contpart}) by $c_1(k,W_n^k(\varepsilon))$, the contribution from $D_u$ or $D_v\in[\varepsilon n/k,n/(\varepsilon k)]$ (area B of Figure~\ref{fig:contpart}) by $c_2(k,W_n^k(\varepsilon))$ and the contribution from $D_u\in[k,k/\varepsilon^2]$ and $D_v\in[\varepsilon^3n/k,\varepsilon n/k]$ (area C of Figure~\ref{fig:contpart}) by $c_3(k,W_n^k(\varepsilon))$. We first study the contribution of area I. In this situation, $D_u,D_v<\varepsilon n/k$, so that we can Taylor expand the exponentials $\me^{-kD_u/L_n}$ and $\me^{-kD_v/L_n}$ in~\eqref{eq:ckexp}. This results in
	\begin{equation}\label{eq:c1exp}
	\begin{aligned}[b]
	\Expn{c_1(k,W_n^k(\varepsilon))} & =\frac{1}{k^2}\sum_{\substack{u,v: (D_u,D_v)\in W_n^{k}(\varepsilon),\\ D_u\in[k/\varepsilon^2,\varepsilon n/k]}} (1-\e^{-kD_u/L_n})(1-\e^{-kD_v/L_n})(1-\e^{-D_uD_v/L_n}) \\
	& = (1+o_\prob(1))\sum_{\substack{u,v: (D_u,D_v)\in W_n^{k}(\varepsilon),\\ D_u\in[k/\varepsilon^2,\varepsilon n/k]}} \frac{D_uD_v}{L_n^2}(1-\e^{-D_uD_v/L_n}).                      
	\end{aligned}
	\end{equation}
	Now we define the random measure 
	\begin{equation}
	\Mn_1[a,b] = \frac{(\mu n)^{\tau-1}}{\log(\varepsilon^3n/k^2)n^2}\sum_{u,v\in[n]}\ind{D_uD_v\in \mu n[a,b], D_u\in[k/\varepsilon^2,\varepsilon n/k]}.
	\end{equation} 
	As similar reasoning as in~\eqref{eq:convbin} shows that
	\begin{equation}
		\Mn_1[a,b]\plim C^2\int_{a}^{b}t^{-\tau}\dd t := \lambda_2[a,b].
	\end{equation}
	By~\eqref{eq:c1exp}, we can write the contribution to the expected value of $c(k)$ in this regime as
	\begin{equation}
	\begin{aligned}[b]
	\Expn{c_1(k,W_n^k(\varepsilon))} &=(1+\op(1))
	\sum_{\substack{u,v: (D_u,D_v)\in W_n^{k}(\varepsilon),\\ D_u\in[k/\varepsilon^3,\varepsilon n/k]}} \frac{D_uD_v}{L_n^2}(1-\e^{-D_uD_v/L_n}) \\
	& =(1+\op(1)) \mu^{-\tau}n^{2-\tau}\log(\varepsilon^3n/k^2)\int_{\varepsilon}^{1/\varepsilon}t(1-\me^{-t})\dd \Mn_1 (t) .                    
	\end{aligned}
	\end{equation}
	Thus, by Lemma~\ref{lem:convmeas}
	\begin{equation}
		\Expn{c_1(k,W_n^k(\varepsilon))} =(1+\op(1)) 2 \mu^{-\tau}n^{2-\tau}\log(\varepsilon^3n/k^2)\int_{\varepsilon}^{1/\varepsilon}t(1-\me^{-t})\dd \lambda_2(t).
	\end{equation}
	
	Then we study the contribution of area B in Figure~\ref{fig:contpart}. This area consists of two parts, the part where $D_u\in [\varepsilon n/k, n/(k\varepsilon)]$, and the part where $D_v\in [\varepsilon n/k, n/(k\varepsilon)]$. By symmetry, these two contributions are the same and therefore we only consider the case where $D_u\in[\varepsilon n/k, n/(k\varepsilon)]$. Then, we can Taylor expand $\me^{-D_vk/L_n}$ in~\eqref{eq:ckexp}, which yields
	\begin{equation}
	\begin{aligned}[b]
		\Expn{c_2(k,W_n^k(\varepsilon))} & =\frac{2}{k^2}\sum_{\substack{u,v: (D_u,D_v)\in W_n^k(\varepsilon),\\D_u>\varepsilon n/k}} (1-\e^{-kD_u/L_n})\frac{D_v k}{L_n}(1-\e^{-D_uD_v/L_n}) .
		\end{aligned}
	\end{equation} 
	Define the random measure
	\begin{equation}
	\Mn_2([a,b],[c,d]):=\frac{(\mu n)^{\tau-1}}{n^2} \sum_{u,v\in[n]}\ind{D_uD_v\in\mu n[a,b], D_u\in (\mu n/k)[c,d]}.
	\end{equation}
	Then we obtain
	\begin{equation}
	\begin{aligned}[b]
	\Expn{c_2(k,W_n^k(\varepsilon))} 
	& =\frac{2}{k L_n}\sum_{\substack{u,v: (D_u,D_v)\in W_n^k(\varepsilon),\\D_u>\varepsilon n/k}}\frac{L_n}{D_u k}(1-\e^{-kD_u/L_n})\frac{D_uD_v k}{L_n}(1-\e^{-D_uD_v/L_n}) \\
	& =2\mu^{-\tau}n^{2-\tau}\int_{\varepsilon}^{1/\varepsilon}\int_{\varepsilon}^{1/\varepsilon}\frac{t_1}{t_2}(1-\me^{-t_1})(1-\me^{-t_2})\dd \Mn_2(t_1,t_2) (1+o_\prob(1)).           
	\end{aligned}
	\end{equation}	
	Again, using a first moment method and a second moment method, we can show that
	\begin{equation}
	\Mn_2([a,b],[c,d])\plim C^2\int_{a}^{b}t^{-\tau}\dd t \int_{c}^{d}\frac 1v \dd v=: \lambda [a,b]\nu[c,d].
	\end{equation}
	Very similarly to the proof of Lemma~\ref{lem:convmeas2} we can show that
	\begin{equation}
	\begin{aligned}[b]
	\int_{\varepsilon}^{1/\varepsilon}\int_{\varepsilon}^{1/\varepsilon}\frac{t_1}{t_2}(1-\me^{-t_1})(1-\me^{-t_2})\dd \Mn_2(t_1,t_2)\plim           \int_{\varepsilon}^{1/\varepsilon}\int_{\varepsilon}^{1/\varepsilon}\frac{t_1}{t_2}(1-\me^{-t_1})(1-\me^{-t_2})\dd \lambda(t_1)\dd \nu(t_2) .
	\end{aligned}
	\end{equation}
	The latter integral can be written as
	\begin{equation}
	\begin{aligned}[b]
	\int_{\varepsilon}^{1/\varepsilon} \int_{\varepsilon}^{1/\varepsilon}\frac{t_1}{t_2}(1-\me^{-t_1})(1-\me^{-t_2})\dd \lambda(t_1)\dd \nu(t_2)  & = C^2\int_{\varepsilon}^{1/\varepsilon}\int_{\varepsilon}^{1/\varepsilon}t_2^{-2}t_1^{1-\tau}(1-\me^{-t_2})(1-\me^{-t_1})\dd t_1\dd t_2\\
	& = C^2\int_{\varepsilon}^{1/\varepsilon}\frac{1}{t_2^2}(1-\me^{-t_2})\dd t_2\int_{\varepsilon}^{1/\varepsilon}t_1^{1-\tau}(1-\me^{-t_1})\dd t_1.
	\end{aligned}
	\end{equation}
	The left integral results in
	\begin{equation}
	\begin{aligned}[b]
	\int_{\varepsilon}^{1/\varepsilon}\frac{1}{t_2^2}(1-\me^{-t_2})\dd t_2 & = \left[\frac{\me^{-t_2}-1}{t_2}+\text{Ei}(t_2)\right]_{t_2=\varepsilon}^{t_2=1/\varepsilon}\\
	& =\varepsilon(\me^{-1/\varepsilon}-1)-\frac{\me^{-\varepsilon}-1}{\varepsilon}+\int_{1/\varepsilon}^{\infty}\frac{1}{u}\me^{-u}\dd u-\log(\varepsilon)-\sum_{j=1}^{\infty}\frac{\varepsilon^k}{k! k}\\
	& =\log\left(\frac{1}{\varepsilon}\right) +\int_{1/\varepsilon}^{\infty}\frac{1}{u}\me^{-u}\dd u+\varepsilon(\me^{-1/\varepsilon}-1)-\frac{\me^{-\varepsilon}-1}{\varepsilon}-\sum_{j=1}^{\infty}\frac{\varepsilon^k}{k! k}\\
	&=\log\left(\frac{1}{\varepsilon}\right) +f(\varepsilon),
	\end{aligned}
	\end{equation}
	where Ei denotes the exponential integral and we have  used the Taylor series for the exponential integral. We can show that $f(\varepsilon)<\infty$ for fixed $\varepsilon\in(0,\infty)$. In fact, $f(\varepsilon)\to 1$ as $\varepsilon\to 0$.
	
	Finally, we study the contribution of area III in Figure~\ref{fig:contpart}, where $D_u\in[k,k/\varepsilon^2]$ and $D_v\in [n/k\varepsilon^3,n/k\varepsilon]$. In this regime, $D_uk\ll n$ and $D_vk\ll n$so that we can Taylor expand the first two exponentials in~\eqref{eq:ckexp}. This results in
	\begin{equation}
		\Expn{c_3(k,W_n^k(\varepsilon))} =(1+o(1))\sum_{\substack{u,v: D_v\in[\varepsilon^3 n/k,\varepsilon n/k], D_uD_v>\varepsilon n,\\ D_u\in[k,k/\varepsilon^2]}}(1-\e^{-D_uD_v/L_n})\frac{D_uD_v}{L_n} .
	\end{equation}
	We define the random measure
	\begin{equation}
		\Mn_3([a,b],[c,d]):=\frac{(\mu n)^{\tau-1}}{n^2}\sum_{u,v}\ind{D_u\in \sqrt{\mu}k[a,b],D_v\in (\sqrt{\mu} n/k)[c,d]}.
	\end{equation}
	Then,
	\begin{equation}
	\begin{aligned}[b]
	\Expn{c_3(k,W_n^k(\varepsilon))} 
	& =(1+o_\prob(1)) n^{2-\tau}\mu^{-\tau}\int_{1}^{1/\varepsilon^2}\int_{\varepsilon/t_1}^{\varepsilon}(t_1t_2)(1-\me^{-t_1t_2})\dd \Mn_3(t_1,t_2) .                     
	\end{aligned}
	\end{equation}
	Again using a first moment method and a second moment method we can show that
	\begin{equation}
		\Mn_3([a,b],[c,d])\plim C^2\int_{a}^{b}u^{-\tau}\dd u\int_{c}^{d}v^{-\tau}\dd v.
	\end{equation}
In a similar way, we can show that for $B\subseteq [1,1/\varepsilon^2]\times [\varepsilon^3,\varepsilon]$, $\Mn_3(B)\plim C^2\int \int_B(uv)^{-\tau}\dd u\dd v$.  
	Thus, by Lemma~\ref{lem:convmeas2},
	\begin{equation}
		\int_{1}^{1/\varepsilon^2}\int_{\varepsilon/t_1}^{\varepsilon}(t_1t_2)(1-\me^{-t_1t_2})\dd \Mn_3(t_1,t_2)\plim C^2\int_{1}^{1/\varepsilon^2}\int_{\varepsilon/x}^{\varepsilon}(xy)^{1-\tau}(1-\me^{-xy})\dd y \dd x.
	\end{equation}
We evaluate the latter integral as
	\begin{equation}
		\begin{aligned}[b]
			 \int_{1}^{1/\varepsilon^2}\int_{\varepsilon/x}^{\varepsilon}(xy)^{1-\tau}(1-\me^{-xy}) \dd y \dd x& =  \int_{1}^{1/\varepsilon^2}\int_{\varepsilon}^{\varepsilon v}\frac{1}{v}u^{1-\tau}(1-\me^{-u})\dd u\dd v \\
			 & =  \int_{\varepsilon}^{1/\varepsilon}\int_{u/\varepsilon}^{1/\varepsilon^2}\frac{1}{v}u^{1-\tau}(1-\me^{-u})\dd v\dd u\\
			 & =  \log\left(\frac{1}{\varepsilon}\right)\int_{\varepsilon}^{1/\varepsilon}u^{1-\tau}(1-\me^{-u})\dd u\\
			 &\quad+ \int_{\varepsilon}^{1/\varepsilon}\log\left(\frac{1}{u}\right)u^{1-\tau}(1-\me^{-u})\dd u.
		\end{aligned}
	\end{equation}
Summing all three contributions to the expectation under $\expec_n$ of the clustering coefficient yields
	\begin{equation}
	\begin{aligned}[b]
	\Expn{c(k,W_n^k(\varepsilon))} &= \Expn{c_1(k,W_n^k(\varepsilon))}+\Expn{c_2(k,W_n^k(\varepsilon))} +\Expn{c_3(k,W_n^k(\varepsilon))} \\
	&=  C^2\mu^{-\tau} n^{2-\tau}(1+\op(1))\Bigg[\int_{\varepsilon}^{1/\varepsilon}t_1^{1-\tau}(1-\me^{-t_1})\dd t_1\\
	&\quad \times \left(\log\left(\frac{n\varepsilon^2}{k^2}\right)+3\log\left(\frac{1}{\varepsilon}\right)+2 f(\varepsilon)\right) +\int_{\varepsilon}^{1/\varepsilon}\log\left(\frac{1}{u}\right)u^{1-\tau}(1-\me^{-u})\dd u\Bigg]\\
	& =C^2(1+\op(1))\mu^{-\tau} n^{2-\tau}\Bigg[\int_{\varepsilon}^{1/\varepsilon}t_1^{1-\tau}(1-\me^{-t_1})\dd t_1\left(\log\left(\frac{n}{k^2}\right)+2 f(\varepsilon)\right)\\
	&\quad +\int_{\varepsilon}^{1/\varepsilon}\log\left(\frac{1}{u}\right)u^{1-\tau}(1-\me^{-u})\dd u\Bigg].
	\end{aligned}
	\end{equation}
Dividing by $n^{2-\tau}\log(n/k^2)$ and taking the limit of $n\to\infty$ then shows that
	\begin{equation}
	\frac{\Expn{c(k,W_n^k(\varepsilon))}}{n^{2-\tau}\log(n/k^2)}\plim C^2\mu^{-\tau}\int_{\varepsilon}^{1/\varepsilon}x^{1-\tau}(1-\me^{-x})\dd x.
	\end{equation}
	\end{proof}
	

	\begin{lemma}\label{lem:convklarge}
		\textup{(Range III)} For $k\gg \sqrt{n}$,
		\begin{equation}
			\frac{\Expn{c(k,W_n^k(\varepsilon))}}{n^{5-2\tau}k^{2\tau-6}}\plim C^2\mu^{3-2\tau}\left(\int_{\varepsilon}^{1/\varepsilon}t^{1-\tau}(1-\me^{-t})\dd t\right)^2.
		\end{equation}
	\end{lemma}
	\begin{proof}
		When $k\gg\sqrt{n}$, the major contribution is from $u$, $v$ with $D_u,D_v=\Theta(n/k)$, so that $D_uD_v=o(n)$. Therefore, we can Taylor expand the exponential $\me^{-D_uD_v/L_n}$ in~\eqref{eq:ckexp}. Thus, we can write the expected value of $c(k)$ as  
	\begin{equation}\label{eq:cklargeexp}
	\begin{aligned}[b]
	\Expn{c(k,W_n^k(\varepsilon))} & =\frac{1}{k^2}\sum_{{u,v: D_u,D_v\in W_n^k(\varepsilon)}} (1-\e^{-kD_u/L_n})(1-\e^{-kD_v/L_n})(1-\e^{-D_uD_v/L_n})        \\
	& = \frac{1}{k^2}\sum_{u,v: D_u,D_v\in W_n^k(\varepsilon)} (1-\e^{-kD_u/L_n})(1-\e^{-kD_v/L_n})\frac{D_uD_v}{L_n}(1+\op(1)).
	\end{aligned}
	\end{equation}
	Define the random measure 
	\begin{equation}
		\Nn_1[a,b]=\frac{(\mu n)^{\tau-1}}{n}k^{1-\tau} \sum_{u\in [n]} \indic{D_u\in (\mu n/k)[a,b]},
	\end{equation}
	and let $\Nn$ be the product measure $\Nn_1\times\Nn_1$. 
	Since all degrees are i.i.d.\ samples from a power-law distribution, the number of vertices with degrees in interval $[q_1,q_2]$ is distributed as a $\text{Bin}(n,C(q_1^{1-\tau}-q_2^{1-\tau}))$ random variable. 
	Therefore,
	\begin{equation}\label{eq:lambda}
	\begin{aligned}[b]
	\Nn_1\left([a,b]\right) & =\frac{(\mu n)^{\tau-1}k^{1-\tau}}{n}\abs{\{i: D_i\in(\mu n/k)[a,b]\}}\plim \lim_{n\to\infty}(\mu n)^{\tau-1}k^{1-\tau}\Prob{D_i\in(\mu n/k)[a,b]} \\
	& = (\mu n)^{\tau-1}k^{1-\tau}\int_{a\mu n/k}^{b\mu n/k} Cx^{-\tau}\dd x=C\int_{a}^{b}t^{-\tau}\dd t:=\lambda([a,b]),                             
	\end{aligned}
	\end{equation}
	where we have used the substitution $t=xk/(\mu n)$. 
	Then,
	\begin{equation}
	\begin{aligned}[b]
	\sum_{u,v: D_u,D_v\in W_n^k(\varepsilon)} & (1-\e^{-kD_u/L_n})(1-\e^{-kD_v/L_n})\frac{D_uD_v}{L_n}                                                                                                            \\
	& =\frac{L_n}{k^2} \sum_{u,v: D_u,D_v\in W_n^k(\varepsilon)}  (1-\e^{-kD_u/L_n})(1-\e^{-kD_v/L_n})\frac{D_u k}{L_n}\frac{D_v k}{L_n}                                                                                                            \\
	& =(1+\op(1))\mu^{3-2\tau} n^{5-2\tau}k^{2\tau-4}\int_{\varepsilon}^{1/\varepsilon}\int_{\varepsilon}^{1/\varepsilon} t_1t_2(1-\me^{-t_1})(1-\me^{-t_2})\dd \Nn(t_1,t_2).                   
	\end{aligned}
	\end{equation}
	Combining this with~\eqref{eq:cklargeexp} yields
	\begin{equation}\label{eq:cklarge1}
	\begin{aligned}[b]
	\frac{\Expn{c(k,W_n^k(\varepsilon))}}{n^{5-2\tau}k^{2\tau-4}} & =(1+\op(1))\mu^{2\tau-3}\int_{\varepsilon}^{1/\varepsilon}\int_{\varepsilon}^{1/\varepsilon} t_1t_2(1-\me^{-t_1})(1-\me^{-t_2})\dd \Nn(t_1,t_2) \\
	& =(1+\op(1))\mu^{2\tau-3}\left(\int_{\varepsilon}^{1/\varepsilon} t_1(1-\me^{-t_1})\dd \Nn_1(t_1)\right)^2.                                      \\
	\end{aligned}
	\end{equation}
	We then use Lemma~\ref{lem:convmeas}, which shows that 
	\begin{equation}\label{eq:cklarge2}
	\int_{\varepsilon}^{1/\varepsilon} t_1^{1-\tau}(1-\me^{-t_1})\dd \Nn_1(t_1) \plim C\int_{\varepsilon}^{1/\varepsilon} t_1(1-\me^{-t_1})\dd \lambda(t_1) = C \int_{\varepsilon}^{1/\varepsilon} t_1^{1-\tau}(1-\me^{-t_1})\dd t_1.
	\end{equation}
	Then, we can conclude from~\eqref{eq:cklarge1} and~\eqref{eq:cklarge2} that
	\begin{equation}
	\frac{\Expn{c(k,W_n^k(\varepsilon)}}{n^{5-2\tau}k^{2\tau-6}}\plim C^2\mu^{3-2\tau}\left( \int_{\varepsilon}^{1/\varepsilon} t_1^{1-\tau}(1-\me^{-t_1})\dd t_1\right)^2.
	\end{equation}
	\end{proof}

	\subsection{Variance of the local clustering coefficient}\label{sec:var}
	In the following lemma, we give a bound on the variance of $c(k,W_n^k(\varepsilon))$:
	\begin{lemma}\label{lem:condvar}
		For all ranges, under $J_n$,
		\begin{equation}
			\frac{\Varn{c(k,W_k^k(\varepsilon))}}{\Expn{c(k,W_n^k(\varepsilon))}^2}\plim  0.
		\end{equation}
	\end{lemma}
\begin{proof}
	We will analyze the variance in a very similar way as we have analyzed the expected value of $c(k)$ conditioned on the degrees in Section~\ref{sec:expmajor}.
	We can write the variance of $c(k,W_n^k(\varepsilon))$ as
	\begin{equation}\label{eq:varcond}
		\begin{aligned}
			&\Varn{c(k,W_n^k(\varepsilon))}\\
			\quad&= \frac{1}{k^2(k-1)^2N_k^2}\sum_{i,j\colon \Der_i,\Der_j=k}\sum_{(u,v),(w,z)\in W_n^k(\varepsilon)}\Probn{\triangle_{i,u,v}\triangle_{j,w,z}}-\Probn{\triangle_{i,u,v}}\Probn{\triangle_{j,w,z}}.
		\end{aligned}
	\end{equation}
	Equation~\eqref{eq:varcond} splits into various cases, depending on the size of $\{i,j,u,v,w,z\}$. We denote the contribution of $\abs{\{i,j,u,v,w,z\}}=r$ to the variance by $V^{\sss{(r)}}(k)$. 
	We first consider $V^{\sss{(6)}}(k)$.
	By a similar reasoning as~\eqref{eq:probtriang}
	\begin{equation}
		\begin{aligned}[b]
		\Varn{c(k,W_n^k(\varepsilon))} &  = \frac{1}{N_k^2k^2(k-1)^2}\sum_{i,j:\Der_i,\Der_j=k}\sum_{(u,v),(w,z)\in W_n^k(\varepsilon)}g_n(k,D_u,D_v)g_n(k,D_w,D_z)(1+\op(1)) \\
		&\qquad\quad -  g_n(k,D_u,D_v)g_n(k,D_w,D_z)(1+\op(1)) \\
		& = \sum_{(u,v),(w,z)\in W_n^k(\varepsilon)}\op\left(\frac{g_n(k,D_u,D_v)g_n(k,D_w,D_z)}{k^2(k-1)^2}\right)	\\
		& = \op\left(\Expn{c(k,W_n^k(\varepsilon))}^2\right),
		\end{aligned}
	\end{equation}
	where we have again replaced $g_n(D_i,D_u,D_v)$ by $g_n(k,D_u,D_v)$ because of~\eqref{eq:gner}.
	 Since there are no overlapping edges when $\{i,j,u,v,w,z\}=5$, $V^{\sss{(5)}}(k)$ can be bounded similarly. This already shows that the contribution to the variance from 5 or 6 different vertices involved is small in all three ranges of $k$.
	
	We then consider the contribution from $V^{\sss{(4)}}$, which is the contribution from two triangles where one edge overlaps. We show that these types of overlapping triangles are rare, so that their contribution to the variance is small.  If for example $i=j$ and $u=z$, then one edge from the vertex of degree $k$ overlaps with another triangle. To bound this contribution, we use that $\Probn{\hat{X}_{ij}=1}\leq \min(1,\frac{D_iD_j}{L_n})$. Then we can bound the summand in~\eqref{eq:varcond} as
	\begin{equation}\label{eq:ptriangij}
		\begin{aligned}[b]
			\Probn{\triangle_{i,u,v}\triangle_{i,w,u}}&-\Probn{\triangle_{i,u,v}}\Probn{\triangle_{i,w,u}}\\
			& \leq \Probn{\triangle_{i,u,v}\triangle_{i,w,u}}\\
			& \leq \min\left(1,\frac{kD_u}{L_n}\right)\min\left(1,\frac{kD_v}{L_n-2}\right)\min\left(1,\frac{D_uD_v}{L_n-4}\right)\\
			&\quad \times \min\left(1,\frac{kD_w}{L_n-6}\right)\min\left(1,\frac{D_wD_u}{L_n-8}\right) \\
			& = (1+\bigO{n^{-1}}) \min\left(1,\frac{kD_u}{L_n}\right)\min\left(1,\frac{kD_v}{L_n}\right)\min\left(1,\frac{D_uD_v}{L_n}\right)\\
			&\quad \times \min\left(1,\frac{kD_w}{L_n}\right)\min\left(1,\frac{D_wD_u}{L_n}\right).                      
		\end{aligned}
	\end{equation}
	We first consider $k$ in Ranges I or II. For the terms involving $k$ we bound this by taking the second term of the minimum, while we bound $\min(D_uD_v/L_n,1)\leq 1$
	\begin{equation}
	\begin{aligned}[b]
	\Probn{\triangle_{i,u,v}\triangle_{i,w,u}}-\Probn{\triangle_{i,u,v}}\Probn{\triangle_{i,w,u}}
	& \leq (1+\bigO{n^{-1}}) \frac{k^3D_uD_vD_w}{L_n^3}\leq O(1)\varepsilon^{-1}\frac{k^3D_w}{L_n^2},
	\end{aligned}
	\end{equation}
	where we used that $D_uD_v<n/\varepsilon$.
	Therefore, the contribution to the variance in this situation can be bounded by
	\begin{equation}
	\begin{aligned}[b]
	\frac{k^3}{k^4 N_k^2}\sum_{i:D_i=k}\sum_{(u,v),(w,u)\in W_n^k(\varepsilon)}\frac{\varepsilon^{-1}D_w}{L_n^2} 
	& = \frac{1}{k N_k}\sum_{(u,v),(w,u)\in W_n^k(\varepsilon)}\frac{\varepsilon^{-1}D_w}{L_n^2} \\
	& \leq   \frac{\varepsilon^{-1}}{k N_k}\sum_u \frac{n}{\varepsilon D_u}\Probn{\mathcal{D}_w>\varepsilon n/D_u\mid D_u}^2 \\
	& \leq K(\varepsilon) \bigOp{\frac{1}{nk^{1-\tau}}\sum_u \left(\frac{n}{D_u}\right)^{3-2\tau}}\\
	& \leq K(\varepsilon)\bigOp{n^{3-2\tau}k^{\tau-1}n^{(\tau-2)/(\tau-1)}},
	\end{aligned}
	\end{equation}
	where we used Lemma~\ref{lem:expind}. Here $K(\varepsilon)$ is a constant only depending on $\varepsilon$. Since $n^{(\tau-2)/(\tau-1)}k^{\tau-1}=O(n)$ when $k<\sqrt{n}$ and $\tau\in(2,3)$, we have proven that this contribution is small enough. 
	Now we consider the contribution from triangles that share the edge between vertices $u$ and $v$. Using a similar reasoning as in~\eqref{eq:ptriangij}, the contribution from the case $i\neq j$ and $u=z$ and $v=w$ can be bounded as
	\begin{equation}
	\begin{aligned}[b]
	\frac{1}{k^4N_k^2}\sum_{i,j:D_i,D_j=k}&\sum_{(u,v)\in W_n^k(\varepsilon)}\Probn{\triangle_{i,u,v}\triangle_{j,v,w}}-\Probn{\triangle_{i,u,v}}\Probn{\triangle_{j,v,w}}\\
	&  \leq \sum_{(u,v)\in W_n^k(\varepsilon)}\frac{k^4D_u^2D_v^2}{k^4L_n^4}\leq \varepsilon^{-2}\Probn{(\mathcal{D}_u,\mathcal{D}_v)\in W_n^k(\varepsilon)}\\
	& = \varepsilon^{-2} \bigOp{n^{1-\tau}\log(n)},
	\end{aligned}
	\end{equation}
	where we used Lemma~\ref{lem:conddeg}. Since $n^{1-\tau}\log(n)=o(n^{4-2\tau}\log^2(n))$ for $\tau\in(2,3)$, this shows that this contribution is small enough.
	 
	When $k$ is in Range III, we use similar bounds for $V^{\sss(4)}$, now using that $D_u,D_v,D_w< \varepsilon^{-1}n/k$. If $N_k=0$, then by definition $\Varn{c(k)}=0$. Therefore, we only consider the case $N_k\geq 1$. 
	Again, we start by considering the case $i=j$ and $u=z$. We can use~\eqref{eq:ptriangij}, where we use that $D_uD_v$ and $D_uD_w<n^2/(k\varepsilon)^2$, and we take 1 for the other minima. This yields
	\begin{equation}
	\Probn{\triangle_{i,u,v}\triangle_{i,w,u}}-\Probn{\triangle_{i,u,v}}\Probn{\triangle_{i,w,u}}\leq O(n^2)k^{-4}\varepsilon^{-4}.
	\end{equation}
	Thus, the contribution to the variance from this case can be bounded as
	\begin{equation}
	\begin{aligned}[b]
	\frac{1}{k^4N_k}\sum_{(u,v),(u,w)\in W_n^k(\varepsilon)}O(n^2)k^{-4}\varepsilon^{-4}
	& \leq\frac{1}{k^4}\bigOp{n^{5}k^{-8}\varepsilon^{-4}\Prob{D>n/(\varepsilon k)}^3}\\
	& \leq \bigOp{n^{5}k^{-8}\varepsilon^{-4}\left(\frac{n}{k\varepsilon}\right)^{3-3\tau}}\\
	& =\bigOp{k^{3\tau-11}n^{8-3\tau}}\varepsilon^{3\tau-7},
	\end{aligned}
	\end{equation}
	where we used Lemma~\ref{lem:conddeg}. 
		When $k>\sqrt{n}$ and $\tau\in(2,3)$, this contribution is smaller that $n^{10-4\tau}k^{4\tau-12}$, as required. 
	In the case where $i\neq j$, $u=z$ and $v=w$,we use a similar reasoning as the one in~\eqref{eq:ptriangij} to show that
	\begin{equation}
		\Probn{\triangle_{i,u,v}\triangle_{i,w,u}}-\Probn{\triangle_{i,u,v}}\Probn{\triangle_{i,w,u}}\leq O(n)k^{-2}\varepsilon^{-2}.
	\end{equation}
	Then the contribution of this situation to the variance can be bounded as
	\begin{equation}
	\frac{1}{k^4}\sum_{(u,v)\in W_n^k(\varepsilon)}O(n)k^{-2}\varepsilon^{-2}\leq \bigO{\varepsilon^{-2}n^3k^{-6}\left(\frac{n}{\varepsilon k}\right)^{2-2\tau}} = \bigO{n^{5-2\tau}k^{2\tau-8}}.
	\end{equation} 
	Again, this is smaller than $n^{10-4\tau}k^{4\tau-12}$, as required. 
	Thus, the contribution of $V^{\sss{(4)}}$ is small enough in all three ranges.

	Finally, $V^{\sss{(3)}}$, can be bounded as
	\begin{equation}\label{eq:V3}
		\frac{1}{k^4N_k^2}\sum_{i:D_i=k}\sum_{u,v:(D_u,D_v)\in W_n^k(\varepsilon)}\Probn{\triangle_{i,u,v}}= \frac{1}{k^4N_k}\Expn{c(k,W_n^k(\varepsilon))}= \frac{1}{k^4N_k}\bigOp{f(k,n)}.
	\end{equation}
	In Ranges I and II, we use that $N_k=\bigOp{nk^{-\tau}}$. Thus, this gives a contribution of
	\begin{equation}
	V^{\sss{(3)}}(k)=\bigOp{\frac{n^{2-\tau}\log(n)}{k^{4-\tau}n}}=\bigOp{n^{1-\tau}\log(n)k^{\tau-4}}, 
	\end{equation}
	which is small enough since $n^{1-\tau}k^{\tau-4}<n^{4-2\tau}$ for $\tau\in(2,3)$ and $k<\sqrt{n}$. 
	In Range III, again we assume that $N_k\geq 1$, since otherwise the variance of $c(k)$ would be zero, and therefore small enough. Then~\eqref{eq:V3} gives the bound
	\begin{equation}
	V^{\sss{(3)}}(k)=\bigOp{n^{5-2\tau}k^{2\tau-10}},
	\end{equation}
	which is again smaller than $n^{10-4\tau}k^{4\tau-12}$ for $\tau\in(2,3)$ and $k\gg \sqrt{n}$. Thus, all contributions to the variance are small enough, which proves the claim. 
\end{proof}

\begin{proof}[Proof of Proposition~\ref{prop:major}.]
	Combining Lemma~\ref{lem:condvar} and the fact that $\Prob{J_n}=1-O(n^{-1/\tau})$ shows that
	\begin{equation}
	\frac{c(k,W_n^k(\varepsilon))}{\Expn{c(k,W_n^k(\varepsilon))}}\plim 1.
	\end{equation}
Then, Lemmas,~\ref{lem:convkmiddle} and~\ref{lem:convklarge} show that
	\begin{equation}
	\frac{c(k,W_n^k(\varepsilon))}{f(k,n)}\plim 
	\begin{cases}
	C^2\int_\varepsilon^{1/\varepsilon}t^{1-\tau}\me^{-t}\dd t+O(\varepsilon^\kappa) & k\ll\sqrt{n}\\
	C^2\left(\int_\varepsilon^{1/\varepsilon}t^{1-\tau}\me^{-t}\dd t\right)^2+O(\varepsilon^\kappa) & k\gg\sqrt{n}.
	\end{cases}
	\end{equation}
which proves the proposition.
\end{proof}

	\section{Contributions outside $W_n^k(\varepsilon)$}\label{sec:minor}
	In this section, we show that the contribution of triangles with degrees outside of the major contributing ranges as described in~\eqref{eq:wkn} is negligible. The following lemma bounds the contribution from triangles with vertices with degrees outside of $W_n^k(\varepsilon)$:
	\begin{lemma}\label{lem:bexp}
		There exists $\kappa>0$ such that
		\begin{equation}
			\limsup_{n\to\infty}\frac{\Expn{c(k,\bar{W}_n^k(\varepsilon))}}{f(n,k)}= \bigOp{\varepsilon^{\kappa}}.
		\end{equation}
	\end{lemma}
	\begin{proof}
		To compute the expected value of $c(k)$, we use that $\Probn{\hat{X}_{ij}=1}\leq \min(1,\frac{D_iD_l}{L_n})$. This yields
	\begin{equation}\label{eq:extriangbound}
	\Expn{c(k)}\leq \frac{n^2\Expn{\min(1,\frac{k\mathcal{D}_u}{L_n})\min(1,\frac{k\mathcal{D}_v}{L_n})\min(1,\frac{\mathcal{D}_u\mathcal{D}_v}{L_n})}}{k(k-1)}.
	\end{equation}
	Using Lemma~\ref{lem:conddeg}, we obtain
	\begin{equation}\label{eq:ckminor}
		\Expn{c(k)}=n^2k^{-2}\bigOp{\Exp{\min\left(1,\frac{k D_u}{\mu n}\right)\min\left(1,\frac{k D_v}{ \mu n}\right)\min\left(1,\frac{D_uD_v}{\mu n}\right)}},
	\end{equation}
	where $D_u$ and $D_v$ are two independent copies of $D$. 
	Similarly,
	\begin{equation}\label{eq:ckWbar}
	\Expn{c(k,\bar{W}_n^k(\varepsilon))}=n^2k^{-2}\bigOp{\Exp{\min\left(1,\frac{k D_u}{\mu n}\right)\min\left(1,\frac{k D_v}{ \mu n}\right)\min\left(1,\frac{D_uD_v}{\mu n}\right)\ind{(D_u,D_v)\in\bar{W}_n^k(\varepsilon)}}},
	\end{equation}
	where
	\begin{equation}\label{eq:expnoncontr}
	\begin{aligned}[b]
		&\Exp{\min\left(1,\frac{k D_u}{\mu n}\right)\min\left(1,\frac{k D_v}{ \mu n}\right)\min\left(1,\frac{D_uD_v}{\mu n}\right)\ind{(D_u,D_v)\in\bar{W}_n^k(\varepsilon)}}\\
		&\quad = \int \int_{(x,y)\in\bar{W}_n^k(\varepsilon)}(xy)^{-\tau}\min\left(1,\frac{k x}{\mu n}\right)\min\left(1,\frac{k y}{ \mu n}\right)\min\left(1,\frac{xy}{\mu n}\right)\dd y \dd x.
	\end{aligned}
	\end{equation}
	We analyze this expression separately for all three ranges of $k$. 
	For ease of notation, we will assume that $\mu=1$ in the rest of this section.
	
	We first consider Range I, where $k\ll n^{(\tau-2)/(\tau-1)}$. Then we have to show that the contribution from vertices $u$ and $v$ such that $D_uD_v<\varepsilon n$ or $D_uD_v>n/\varepsilon$ is small. 
			First, we study the contribution to~\eqref{eq:expnoncontr} for $D_uD_v<\varepsilon n$. 
		We can bound this contribution by taking the second term of the minimum in all three cases, which gives
		\begin{equation}
		\frac{k^2}{n^3}\int_{1}^{n}\int_{1}^{\varepsilon n/x}(xy)^{2-\tau}\dd y \dd x = \frac{k^2}{n^3}\int_{1}^{n}\frac 1x \int_{x}^{\varepsilon n}u^{2-\tau}\dd u \dd x = \frac{k^2\varepsilon^{3-\tau}}{3-\tau}\bigO{n^{-\tau}\log(n)}.
		\end{equation}
		Then, we study the contribution for $D_uD_v>n/\varepsilon$. This contribution can be bounded very similarly by taking $\frac{kD_u}{L_n}$ and $\frac{kD_uv}{L_n}$ and 1 for the minima in~\eqref{eq:expnoncontr}
		\begin{equation}
		\frac{nk^2}{n^2}\int_{1}^{n}\int_{n/(\varepsilon x)}^n(xy)^{1-\tau}\dd y \dd x = \frac{k^2}{n^2}\int_{1}^{n}\frac 1x \int_{n/\varepsilon}^{nx}u^{1-\tau}\dd u \dd x = \frac{k^2\varepsilon^{\tau-2}}{\tau-2} \bigO{n^{-\tau}\log(n)} .
		\end{equation}
		By~\eqref{eq:ckWbar}, 
		\begin{equation}
			\Expn{c(k,\bar{W}_n^k(\varepsilon))}=\bigOp{n^{2-\tau}\log(n)\varepsilon^{\kappa}}.
		\end{equation}
		Multiplying by $n^2k^{-2}$ and dividing by $n^{2-\tau}\log(n)$ and taking the limit for $n\to\infty$ then proves the lemma in Range I by~\eqref{eq:ckminor}. 		
	
	Now we consider Range II, where $an^{(\tau-2)/(\tau-1)}\leq k\ll \sqrt{n}$ for some $a>0$. We show that the contribution from vertices $u$ and $v$ such that $D_uD_v<\varepsilon n$ or $D_uD_v>n/\varepsilon$ or $D_u,D_v>n/(k\varepsilon)$ is small. 
		We first show that the contribution to~\eqref{eq:expnoncontr} for $D_u>n/(k\varepsilon)$ is small. In this setting, $D_uk>n$, so that the first minimum in~\eqref{eq:expnoncontr} is attained by 1. The contribution can be computed as
		\begin{equation}
		\begin{aligned}[b]
			& \int_{n/(k\varepsilon)}^\infty \int_1^\infty(xy)^{-\tau}\min\left(1,\frac{k y}{ n}\right)\min\left(1,\frac{xy}{n}\right)\dd y \dd x\\ & =\frac{k}{n^2}\int_{n/(\varepsilon k)}^{\infty}\int_{1}^{n/x}x^{1-\tau}y^{2-\tau}\dd y\dd x + 	\frac{k}{n}\int_{n/(k\varepsilon)}^{\infty}\int_{n/x}^{n/k}x^{-\tau}y^{1-\tau}\dd y\dd x\\
			& \quad  + 	\int_{n/(k\varepsilon)}^{\infty}\int_{n/k}^{\infty}x^{-\tau}y^{-\tau}\dd y\dd x\\
			&= k^2\bigO{n^{-\tau}}+ k^2\bigO{n^{-\tau}} +  \varepsilon^{\tau-1}\bigO{n^{2-2\tau}k^{2\tau-2}}.
		\end{aligned}
		\end{equation}
		By~\eqref{eq:ckminor}, multiplying by $n^{2}k^{-2}$ and then dividing by $n^{2-\tau}\log(n/k^2)$ and letting $n$ go to infinity shows that this contribution is small. 
		Thus, we may assume that $D_u,D_v<n/(k\varepsilon)$.
		Now we show that the contribution from $D_uD_v<\varepsilon n$ is negligible. Then, $D_uD_v<n$, so that the third minimum in~\eqref{eq:expnoncontr} is attained for $D_uD_v/n$. The contribution then splits into various cases, depending on $D_u$. 
		\begin{equation}
			\begin{aligned}[b]
				& \frac1n \int  \int_{xy<\varepsilon n}(xy)^{1-\tau}\min\left(1,\frac{k x}{ n}\right)\min\left(1,\frac{k y}{n}\right)\dd y \dd x \\
				&=	\int_{1}^{k}\int_{1}^{\varepsilon n/x}(xy)^{-\tau}\frac{kx^2y}{L_n^2}\dd y \dd x +\int_{k}^{n/k}\int_{1}^{\varepsilon n/x}(xy)^{-\tau}\frac{k^2x^2y^2}{L_n^3}\dd y \dd x  + \int_{n/k}^{\infty}\int_{1}^{\varepsilon n/x}(xy)^{-\tau}\frac{kxy^2}{L_n^3}\dd y \dd x\\
				& = k^2\bigO{n^{-\tau}}\varepsilon^{2-\tau}+ k^2 \varepsilon n^{-\tau}\bigO{\log(n/k^2)}+ k^2\bigO{n^{-\tau}}\varepsilon^{3-\tau}.
			\end{aligned}
		\end{equation}
		The contribution of $D_uD_v>n/\varepsilon$ can be bounded similarly as
		\begin{equation}
		\begin{aligned}[b]
			&  \int  \int_{xy>n/\varepsilon }(xy)^{-\tau}\min\left(1,\frac{k x}{ n}\right)\min\left(1,\frac{k y}{n}\right)\dd y \dd x \\
			& = \int_1^{k}\int_{n/(\varepsilon x)}^\infty(xy)^{-\tau}\frac{kx}{L_n}\dd y \dd x + 	\int_{k}^{n/k}\int_{n/(\varepsilon x)}^\infty(xy)^{-\tau}\frac{k^2xy}{L_n^2}\dd y \dd x  + 	\int_{n/k}^\infty\int_{n/(\varepsilon x)}^\infty(xy)^{-\tau}\frac{ky}{L_n}\dd y \dd x\\
			& =  k^2\varepsilon^{\tau-1}\bigO{n^{-\tau}} + k^2\varepsilon^{\tau-2}\bigO{n^{-\tau}\log(n/k^2)} +  k^2\bigO{n^{-\tau}}\varepsilon^{\tau-2}.
		\end{aligned}
		\end{equation}
		By~\eqref{eq:ckWbar}, multiplying by $k^{-2}n^2$ and then dividing by $k(k-1) n^{2-\tau}\log(n/k^2)$ proves the lemma in Range II.

	Finally, we prove the lemma in Range III, where $k\gg\sqrt{n}$. Here we have to show that the contribution from $D_u, D_v<\varepsilon n/k$ or  $D_u,D_v>n/(\varepsilon k)$ is small. 
		We again bound this by using~\eqref{eq:expnoncontr}. 
		The contribution to~\eqref{eq:expnoncontr} for $D_u>n/(k\varepsilon)$ can be computed as
		\begin{equation}
		\begin{aligned}[b]
			& \int_{n/(k\varepsilon)}^\infty \int_1^\infty(xy)^{-\tau}\min\left(1,\frac{k y}{ n}\right)\min\left(1,\frac{xy}{n}\right)\dd y \dd x\\
			& = \int_{\frac{n}{k\varepsilon}}^{k}\int_{n/x}^{\infty}x^{-\tau}y^{-\tau}\dd y\dd x + \int_{\frac{n}{k\varepsilon}}^{k}\int_{n/k}^{n/x}\frac 1n x^{-\tau+1}y^{-\tau+1}\dd y\dd x + 	\int_{\frac{n}{k\varepsilon}}^{k}\int_{0}^{n/k}\frac {k}{n^2} x^{-\tau+1}y^{-\tau+2}\dd y\dd x\\
			& \quad + 	\int_{k}^\infty\int_{n/k}^{\infty}x^{-\tau}y^{-\tau}\dd y\dd x + \int_{k}^\infty\int_{n/x}^{n/k}\frac kn x^{-\tau}y^{-\tau+1}\dd y\dd x + \int_{k}^{\infty}\int_{0}^{n/x}\frac {k}{n^2} x^{-\tau+1}y^{-\tau+2}\dd y\dd x\\
			&=			
			\bigO{\log\left(\frac{k^2\varepsilon}{n}\right)n^{1-\tau}} + \bigO{\varepsilon^{\tau-2}k^{2\tau-4}n^{3-2\tau}}+ \bigO{n^{1-\tau}}+\bigO{\varepsilon^{\tau-2}n^{3-2\tau}k^{2\tau-4} }\\
			&\quad + \bigO{n^{1-\tau}} + \bigO{n^{1-\tau} }+ \bigO{n^{1-\tau}} =  \bigO{\varepsilon^{\tau-2}k^{2\tau-4}n^{3-2\tau}}.
		\end{aligned}
		\end{equation}
		Multiplying this by $n^2k^{-2}$ and then dividing by $n^{5-2\tau}k^{2\tau-6}$ shows that this contribution is small.

		Then we study the contribution to~\eqref{eq:expnoncontr} for $D_1<\varepsilon n/k$. This can be computed as
		\begin{equation}
		\begin{aligned}[b]
			& \frac 1n\int_1^{\varepsilon n/k} \int_1^\infty(xy)^{1-\tau}\min\left(1,\frac{k y}{ n}\right)\min\left(1,\frac{xy}{n}\right)\dd y \dd x\\
			& = \int_0^{\frac{n\varepsilon}{k}}\int_{0}^{n/k}\frac{k^2}{n^3}x^{-\tau+2}y^{-\tau+2}\dd y\dd x + \int_0^{\frac{n\varepsilon}{k}}\int_{n/k}^{n/x}\frac{k}{n^2}x^{-\tau+2}y^{-\tau+1}\dd y\dd x + 
			\int_0^{\frac{n\varepsilon}{k}}\int_{n/x}^{\infty}\frac{k}{n}x^{-\tau+1}y^{-\tau}\dd y\dd x \\
			& = \bigO{\varepsilon^{3-\tau} k^{2\tau-4}n^{3-2\tau} }+ \bigO{\varepsilon^{3-\tau} k^{2\tau-4}n^{3-2\tau}}+ \bigO{\varepsilon n^{1-\tau}} =  \bigO{\varepsilon^{\tau-2}k^{2\tau-4}n^{3-2\tau}}.
		\end{aligned}
	\end{equation}
		Thus, dividing these estimates by $n^{3-2\tau}k^{2\tau-6}$ and noting that $n^{1-\tau}<n^{3-2\tau}k^{2\tau-4}$ for $\sqrt{n}\ll k\ll n$ completes the proof in Range III. 
	\end{proof}

\subsection{Proof of Theorem~\ref{thm:sqrt}}\label{sec:proofthm}
	We now show how we adjust the proof of Theorem~\ref{thm:ck} to prove Theorem~\ref{thm:sqrt}. We use the same major contributing triangles as the ones in Range III in~\eqref{eq:wkn}. Then, in fact Lemmas~\ref{lem:expk},~\ref{lem:condvar} and Proposition~\ref{prop:minor} still hold. It is easy to derive a similar lemma as Lemma~\ref{lem:convklarge} for the situation $k=\Theta(\sqrt{n})$. The only difference with the proof of Lemma~\ref{lem:convklarge} is that we do not Taylor expand the exponentials in~\eqref{eq:cklargeexp}. This then proves Theorem~\ref{thm:sqrt}.\hfill \qed
	
\subsection{Proof of Theorem \ref{theorem3}}\label{sec:proofthm3}
We now prove that the scaling limit of $k\mapsto c(k)$ is continuous around $k=\sqrt{n}$.
When $B$ is large, we rewrite~\eqref{eq:cksqrt} as
\begin{equation}\label{eq:cksqrtlarge}
\frac{c(k)}{n^{2-\tau}}  \plim C^2\mu^{2-2\tau}B^{2\tau-4}\int_{0}^{\infty}\int_{0}^{\infty}(xy)^{-\tau}(1-\me^{-x})(1-\me^{-y})(1-\me^{-xy\mu/B^2})\dd x\dd y.
\end{equation}
Taylor expanding the last exponential then yields
\begin{equation}
\begin{aligned}[b]
\frac{c(k)}{n^{2-\tau}}& = (1+o(1))C^2\mu^{3-2\tau}B^{2\tau-6}\int_{0}^{\infty}\int_{0}^{\infty}(xy)^{1-\tau}(1-\me^{-x})(1-\me^{-y})\dd x\dd y \\
& = (1+o(1))C^2\mu^{3-2\tau}B^{2\tau-6}A^2.
\end{aligned}
\end{equation}
Substituting $k=B\sqrt{n}$ in Range III of Theorem~\ref{thm:ck} gives
\begin{equation}
\frac{c(k)}{n^{2-\tau}}  =(1+o(1)) C^2\mu^{3-2\tau}B^{2\tau-6}A^2,
\end{equation}
which is the same as the result obtained from Theorem~\ref{thm:sqrt}. Therefore, the scaling limit of $k\mapsto c(k)$ is smooth for $k>\sqrt{n}$. 

For $B$ small, we can Taylor expand the first two exponentials in~\eqref{eq:cksqrt} as long as $x,y\ll 1/B$. The contribution where $x,y<1/B$ and $B<xy<1/B$ can be written as 
\begin{equation}
	\begin{aligned}[b]
	& C^2\mu^{2-2\tau}\left(\int_{B^2}^{1}\int_{B/x}^{1/B}(xy)^{1-\tau}(1-\me^{-\mu xy})\dd x\dd y+\int_{1}^{1/B}\int_{B/x}^{1/(Bx)}(xy)^{1-\tau}(1-\me^{-\mu xy})\dd x\dd y\right)\\
	& = C^2\mu^{-\tau}\left(\int_{B^2}^{1}\int_{B}^{v/B}\frac{1}{v}u^{1-\tau}(1-\me^{-u})\dd u\dd v+\int_{1}^{1/B}\int_{B}^{1/B}\frac{1}{v}u^{1-\tau}(1-\me^{-u})\dd u\dd v\right)\\
	& =  C^2\mu^{-\tau}\left(\log(1/B^2)\int_{B}^{1/B}u^{1-\tau}(1-\me^{-u})\dd u+\int_{B}^{1/B}\log(1/u)u^{1-\tau}(1-\me^{-u})\dd u\right).
	\end{aligned}
\end{equation}
The contribution of the second integral becomes small compared to the first part as $B$ gets small, as the second integral is finite for $B>0$. We can show that the contributions from $x,y>1/B$, or from $xy>1/B$ can also be neglected by using that $1-\me^{-x}\leq \min(1,x)$. Thus, as $B$ becomes very small, Theorem~\ref{thm:sqrt} shows that $c(k)$ for $k=B\sqrt{n}$ can be approximated by
\begin{equation}
\frac{c(k)}{n^{2-\tau}}\approx C^2\log(B^{-2})\int_{0}^{\infty}u^{1-\tau}(1-\me^{-u})\dd u,
\end{equation}
which agrees with the value for $k=B\sqrt{n}$ in Range II of Theorem~\ref{thm:ck}.

To prove the continuity around $k=n^{(\tau-2)/(\tau-1)}$, we fill in $k=an^{(\tau-2)/(\tau-1)}$ in Range II of Theorem~\ref{thm:ck}, which yields
\begin{equation}
c(k)=n^{2-\tau}\mu^{-\tau}C^2A\left(\frac{3-\tau}{\tau-1}\log(n)+\log(a^{-2}\right)(1+\op(1)).
\end{equation}
This agrees with the $k\mapsto c(k)$ curve in Range I when $n$ grows large.\hfill \qed
\medskip

\paragraph{\bf Acknowledgement.} This work was supported by NWO TOP grant 613.001.451 and by the NWO Gravitation Networks grant 024.002.003.
The work of RvdH is further supported by the NWO VICI grant 639.033.806.  The work of JvL is further supported by an NWO TOP-GO grant and by an ERC Starting Grant.

\bibliographystyle{abbrv}      
\bibliography{../references}

\end{document}